\documentclass[aop]{imsart}

\RequirePackage{amsthm,amsmath,amsfonts,amssymb}
\RequirePackage[numbers]{natbib}
\RequirePackage[colorlinks,citecolor=blue,urlcolor=blue]{hyperref}
\RequirePackage{graphicx}

\startlocaldefs

\newtheorem{theorem}{Theorem}[section]
\newtheorem*{theorem-non}{Theorem}
\newtheorem{lemma}[theorem]{Lemma}
\theoremstyle{remark}

\newtheorem{question}[theorem]{Question}

\newtheorem{prop}[theorem]{Proposition}

\usepackage{t1enc}
\usepackage[utf8]{inputenc}
\usepackage{amsthm,amsmath,amssymb}
\usepackage{graphicx}
\usepackage{enumerate}
\usepackage{hyperref}
\usepackage{bm}

\usepackage{amsfonts}
\usepackage{graphicx}
\usepackage{bm}
\usepackage{amsmath, amsthm, amssymb}
\usepackage{graphicx}
\usepackage{hyperref}
 \usepackage{relsize}

\usepackage{bbm}

\startlocaldefs

\newcommand{\wpc}{RGPC}
\newcommand{\wpcs}{RGPCs}
\newcommand{\rwo}{RGO}
\newcommand{\rwos}{RGOs}

\DeclareMathOperator{\range}{Im }

\DeclareMathOperator{\Tr}{Tr}

\DeclareMathOperator{\rest}{rest}

\DeclareMathOperator{\Aut}{Aut}

\def\sver{1}

\endlocaldefs

\begin{document}

\begin{frontmatter}
\title{Limiting entropy of determinantal processes}
\runtitle{Limiting entropy of determinantal processes}

\begin{aug}
\author[A]{\fnms{Andr\'as} \snm{M\'eszaros}\ead[label=e1]{Meszaros\_Andras@phd.ceu.edu}},
\address[A]{Central European University, Budapest
\printead{e1}}

\end{aug}

\begin{abstract}
We extend Lyons's tree entropy theorem to general determinantal measures. As a byproduct we show that the sofic entropy of an invariant determinantal measure does not depend on the chosen sofic approximation.

\end{abstract}



\end{frontmatter}

\section{Introduction}
 
Let $P=(p_{ij})$ be an orthogonal projection matrix, where rows and columns are both indexed with a finite set $V$. Then there is a unique probability measure $\eta_P$ on the subsets of $V$ such that for every $F\subset V$ we have 
\[\eta_P(\{B|F\subset B\subset V\})=\det (p_{ij})_{i,j\in F}.\]
The measure $\eta_P$ is called the \emph{determinantal measure} corresponding to $P$ \cite{lydet}. Let $B^P$ be a random subset of $V$ with distribution $\eta_P$.  In this paper we investigate the asymptotic behavior of the    \emph{Shannon-entropy} of $B^P$ defined as
\[H(B^P)=\sum_{A\subset V} -\mathbb{P}(B^P=A)\log \mathbb{P}(B^P=A).\] 

Let $P_1,P_2,\dots$ be a sequence of orthogonal projection matrices. Assume that rows and columns of $P_n$ are both indexed with the finite set~$V_n$. Let $G_n$ be a graph on the vertex~set~$V_n$.  Throughout the paper we assume that  the degrees of  graphs are at most $D$ for some fixed finite $D$. Our main theorem is the following.

\begin{theorem-non}
Assume that the sequence of pairs $(G_n,P_n)$ is Benjamini-Schramm convergent and tight. Then
\[\lim_{n\to\infty} \frac{H(B^{P_n})}{|V_n|}\]
exists.
\end{theorem-non}

Note that this theorem will be restated in a slightly more general and precise form as Theorem \ref{MainThm_t} in the next section. 
We will also give a formula for the limit.

We define  \emph{Benjamini-Schramm convergence} of  $(G_n,P_n)$ along the lines of \cite{besch} and \cite{ally} via the following local sampling procedure. Fix any positive integer $r$, this will be our radius of sight. For a vertex $o\in V_n$ let $B_r(G_n,o)$ be the $r$-neighborhood of $o$  in the graph $G_n$, and let $M_{n,r,o}$ be the submatrix of $P_n$ determined by  rows and columns with indeces in $B_r(G_n,o)$. Then the outcome of the local sampling at $o$ is the pair $(B_r(G_n,o),M_{n,r,o})$. Of course, we are only interested in the outcome up to rooted isomorphism. Now if we pick $o$ as a uniform random element of $V_n$, we get a probability measure $\mu_{n,r}$ on the set of isomorphism classes of pairs $(H,M)$, where $H$ is a rooted $r$-neighborhood and $M$ is a matrix where  rows and columns are indexed with the vertices of $H$. We say that the sequence $(G_n,P_n)$ converges if for any fixed $r$ the measures $\mu_{n,r}$ converge weakly as $n$ tends to infinity.  
 See the next section for more details including the description of the limit object. 

To define the notion of tightness,  we introduce a measure $\nu_n$ on $\mathbb{N}\cup\{\infty\}$ for each pair $(G_n,P_n)$ as follows.  Given $k\in \mathbb{N}\cup\{\infty\}$ we set 

\[\nu_n(\{k\})=|V_n|^{-1}\sum_{\substack{u,v\in V_n\\ d_n(u,v)=k}} |P_n(u,v)|^2,\]
where $d_n$ is the graph metric on $V_n=V(G_n)$. Then the sequence $(G_n,P_n)$ is tight if the family of measures $\nu_n$ is tight, that is, for each $\varepsilon>0$ we have a finite $R$ such that 
\[\nu_n\left(\{R+1,R+2,\dots\}\cup\{\infty\}\right)<\varepsilon\]
for all $n$. Tightness makes sure that the local sampling procedure from the previous paragraph detects most of the significant matrix entries for large enough $r$.  

Note that a related convergence notion of operators was introduced by Lyons and Thom~\cite{lyth}. We expect that their notion is slightly stronger, but were unable to clarify this.
 
The idea of the proof of the main theorem is the following. Consider a uniform random ordering of $V_n$. Then using the chain rule for conditional entropy we can write $H(B^{P_n})$ as the sum of $|V_n|$ conditional entropies. We show that in the limit we can control these conditional entropies. This method in the context of local convergence first appeared in \cite{borgs}. 

Now we describe a special case of our theorem. Consider a finite connected graph $G$, and consider the uniform measure on the set of spanning trees of $G$. This measure turns out to be a determinantal measure, the corresponding projection matrix $P_{\bigstar}(G)$ is called the transfer-current matrix \cite{trans}. Since this is a uniform measure, the Shannon-entropy is simply $\log \tau(G)$, where $\tau(G)$ is the number of spanning trees in $G$. A theorem of Lyons \cite{treeent} states that if $G_n$ is a Benjamini-Schramm convergent sequence of finite connected graphs then
\[\lim_{n\to\infty} \frac{\log\tau(G_n)}{|V(G_n)|}\]
exists. This theorem now follows from our results, because it is easy to see that the sequence $(L(G_n),P_{\bigstar}(G_n))$ is convergent and tight in our sense, where $L(G_n)$ is the line graph of $G_n$. See Section \ref{sec7}. Note that we need to take the line graph of $G_n$, because the uniform spanning tree measure is defined on the edges of $G_n$ rather than the vertices of $G_n$. We also obtain a formula for limit which is different from Lyons's original formula. However, in practice it seems easier to evaluate Lyons's original formula.

\if\sver 1

 Another application comes from ergodic theory. Let $\Gamma$ be a finitely generated countable group, and let $T$ be an invariant positive contraction on $\ell^2(\Gamma)$. Here a linear operator is called a positive contraction if it is positive semidefinite and has operator norm at most $1$. Invariance means that for any $\gamma,g_1,g_2\in \Gamma$  we have
 \[\langle Tg_1,g_2\rangle=\langle T(\gamma^{-1}g_1),\gamma^{-1}g_2\rangle.\]
Note that here we identify elements of $\Gamma$ with their characteristic vectors. 
Then the determinantal  measure $\mu_T$ corresponding to $T$ gives us an invariant measure on $\{0,1\}^\Gamma$. Note that there is a natural graph structure on $\Gamma$. Namely, we can fix a finite generating set $S$, and consider the corresponding Cayley-graph $\text{Cay}(\Gamma,S)$.  When $\Gamma$ belongs to the class of sofic groups, one can define the so-called sofic entropy of this invariant measure \cite{abwe}. This is done by first considering an approximation of $\text{Cay}(\Gamma,S)$ by a sequence of finite graphs $G_n$, and then investigating how  we can  model $\mu_T$ on these finite graphs. In general it is not known whether sofic entropy depends on the chosen approximating sequence $G_n$ or not, apart from certain trivial examples. However, in our special case, our results allow us to give a formula for the sofic entropy, which only depends on the measure $\mu_T$, but not on the finite approximations. This shows that in this case the sofic entropy does not depend on the chosen sofic approximation.    

\else
Another application comes from ergodic theory. Let $Q$ be a vertex transitive connected infinite graph which is a Benjamini-Schramm limit of finite graphs, and let $T$ be an invariant positive contraction on $\ell^2(V(Q))$. Here a linear operator is called a positive contraction if it is positive semidefinite and has operator norm at most $1$. Invariance means that for any $\gamma\in \Aut(Q)$ and $v_1,v_2\in V(Q)$  we have
\[\langle Tv_1,v_2\rangle=\langle T(\gamma^{-1} v_1),\gamma^{-1} v_2\rangle.\footnote{Here we identify elements of $\Gamma$ with their characteristic vectors.}\]
Then the determinantal  measure $\mu_T$ corresponding to $T$ gives us an invariant measure on $\{0,1\}^{V(Q)}$. 
One can define the so-called sofic entropy of this invariant measure \cite{abwe}. This is done by first considering an approximation of $Q$ by a sequence of finite graphs $G_n$, and then investigating how  we can  model $\mu_T$ on these finite graphs. In general it is not known whether sofic entropy depends on the chosen approximating sequence $G_n$ or not, apart from certain trivial examples. However, in our special case, our results allow us to give a formula for the sofic entropy, which only depends on the measure $\mu_T$, but not on the finite approximations. This shows that in this case the sofic entropy does not depend on the chosen sofic approximation.    
\fi

Observe that in our main theorem the graphs $G_n$ do not play any role in the definition of the random subsets $B^{P_n}$ or the Shannon entropy $H(B^{P_n})$, they are only there to help us define our convergence notion. This suggests that there might be a notion of convergence of orthogonal projection matrices without any additional graph structure such that the normalized Shannon entropy of $B^{P_n}$ is  continuous.


\textbf{Structure of the paper.}
 In Section \ref{sec2} we explain the basic definitions and state our results. 
In Section \ref{sec3} we investigate what happens if we condition a Benjamini-Schramm convergence sequence of determinantal measures in a Benjamini-Schramm convergent way. 
In Sections \ref{sec4}, \ref{sec5} and \ref{sec6} we prove the theorems stated in Section \ref{sec2}. 
In Section \ref{sec7} we explain the connections of our results and Lyons's tree entropy theorem. The proof of a  technical lemma about the measurability of the polar decomposition is given in the Appendix.         

\section{Definitions and statements of the results}\label{sec2}
\subsection{The space of rooted graphs and sofic groups}
Fix a degree bound $D$. A \emph{rooted graph} is a pair $(G,o)$ where $G$ is a (possibly infinite) connected graph with degrees at most $D$, $o\in V(G)$ is a distinguished vertex of $G$ called the root. Given two rooted graphs $(G_1,o_1)$ and $(G_2,o_2)$ their distance is defined to be the infimum over all $\varepsilon>0$ such that for $r=\lfloor \varepsilon^{-1}\rfloor$ there is a root preserving graph isomorphism  from $B_r(G_1,o_1)$ to $B_r(G_2,o_2)$. Let $\mathcal{G}$ be the set of isomorphism classes of rooted graphs. With the above defined distance $\mathcal{G}$ is a compact metric space. Therefore, the set of probability measures $\mathcal{P}(\mathcal{G})$ endowed with the weak* topology is also compact. A sequence of random rooted graphs $(G_n,o_n)$ \emph{Benjamini-Schramm converges} to the random rooted graph $(G,o)$, if their distributions converge in $\mathcal{P}(\mathcal{G})$. Given any finite graph $G$, we can turn it into a random rooted graph $U(G)=(G_o,o)$ by considering a uniform random vertex $o$ of $G$ and its connected component $G_o$. A sequence of finite graphs $G_n$ Benjamini-Schramm converges to the random rooted graph $(G,o)$ if the sequence $U(G_n)$ Benjamini-Schramm converges to $(G,o)$.   

Let $S$ be a finite set, an \emph{$S$-labeled Schreier graph} is a graph where each edge is oriented and labeled with an element from $S$, moreover for every vertex $v$ of the graph and every $s\in S$ there is exactly one edge labeled with $s$ entering $v$ and there is exactly one edge labeled with $s$ leaving $v$. For example, if $\Gamma$ is a group with generating set $S$, then its Cayley-graph $\text{Cay}(\Gamma,S)$ is an $S$-labeled Schreier-graph. The notion of Benjamini-Schramm convergence can be extended to the class of $S$-labeled Schreier-graphs with the modification that graph isomorphisms are required to respect the orientation and labeling of the edges.  Let $\Gamma$ be a finitely generated group. Fix a finite generating set $S$, and consider the Cayley-graph $G_\Gamma=\text{Cay}(\Gamma,S)$. Let $e_\Gamma$ be the identity of $\Gamma$. We say that $\Gamma$ is \emph{sofic} if there is a sequence of finite $S$-labeled Schreier-graphs $G_n$, such that $G_n$ Benjamini-Schramm converges to $(G_\Gamma,e_\Gamma)$.

\subsection{The space of rooted graph-operators}

Fix a degree bound $D$, and let $K$ be a non-empty finite set.

A \emph{rooted graph-operator} (\rwo) is a triple $(G,o,T)$, where $(G,o)$ is a rooted graph and $T$ is a bounded operator on $\ell^2(V(G)\times K)$. In this paper we will use real Hilbert spaces, but the results can be generalized to the complex case as well. Note that to prove our main theorem it suffices to only consider the case $|K|=1$. The usefulness of allowing $|K|>1$ will be only clear in Section \ref{sec5}, where we extend our results to positive contractions.

Given two \rwos\  $(G_1,o_1,T_1)$ and $(G_2,o_2,T_2)$ their distance  $d((G_1,o_1,T_1),(G_2,o_2,T_2))$ is 
defined as the infimum over all $\varepsilon>0$ such that for $r=\lfloor \varepsilon^{-1}\rfloor$ there is a root preserving graph isomorphism $\psi$ from $B_r(G_1,o_1)$ to $B_r(G_2,o_2)$ with the property that 
\begin{equation}\label{ineq1}
|\langle T_1(v,k), (v',k')\rangle-\langle T_2(\psi(v),k), (\psi(v'),k')\rangle|<\varepsilon
\end{equation} 
for every $v,v'\in V(B_r(G_1,o_1))$ and $k,k'\in K$. Here we identified elements of $V(G_i)\times K$ with their characteristic vectors in $\ell^2(V(G_i)\times K)$. 

Two \rwos\  $(G_1,o_1,T_1)$ and $(G_2,o_2,T_2)$ are called isomorphic if their distance is $0$, or equivalently if there is a root preserving graph isomorphism $\psi$ from $(G_1,o_1)$ to $(G_2,o_2)$ such that \[\langle T_1(v,k),(v',k')\rangle=\langle T_2(\psi(v),k), (\psi(v'),k')\rangle\] for every $v,v'\in V(G_1)$ and $k,k'\in K$. Let $\mathcal{\rwo}$ be the set of isomorphism classes of \rwos. For any $0<B<\infty$, we define
\[\mathcal{\rwo}(B)=\{(G,o,T)\in \mathcal{\rwo}|\quad \|T\|\le B\}.\]
One can prove that $\mathcal{\rwo}(B)$ is a compact metric space with the above defined distance $d$. Let $\mathcal{P}(\mathcal{\rwo}(B))$ be the set of probability measures on $\mathcal{\rwo}(B)$ endowed with the weak* topology, this is again a compact space. Often it will be more convenient to consider  an element $\mathcal{P}(\mathcal{\rwo})$ as a random \rwo.

A \rwo\ $(G,o,T)$ is called a \emph{rooted graph-positive-contraction} (\wpc) if $T$ is a self-adjoint positive operator with norm at most $1$. Then the set $\mathcal{\wpc}$ of isomorphism classes of \wpcs\ is a compact metric space. Therefore, $\mathcal{P}(\mathcal{\wpc})$ with the weak* topology is compact.

We need a slight generalization of the notion of \rwo. An \emph{$h$-decorated \rwo} is a tuple  $(G,o,T,A^{(1)},A^{(2)},\dots,A^{(h)})$, where $G,o$ and $T$ are like above, $A^{(1)},A^{(2)},\dots,A^{(h)}$ are subsets of $V(G)\times K$. Given two $h$-decorated \rwos\  $(G_1,o_1,T_1,A^{(1)}_1,A^{(2)}_1,\dots,A^{(h)}_1)$ and  $(G_2,o_2,T_2,A^{(1)}_2,A^{(2)}_2,\dots,A^{(h)}_2)$ their distance is 
defined as the infimum over all $\varepsilon>0$ such that for $r=\lfloor \varepsilon^{-1}\rfloor$ there is a root preserving graph isomorphism $\psi$ from $B_r(G_1,o_1)$ to $B_r(G_2,o_2)$ satisfying the property given in (\ref{ineq1}), and for $i=1,2,\dots,h$ we have \[\bar{\psi}(A_1^{(i)}\cap(B_r(G_1,o_1)\times K))=A_2^{(i)}\cap(B_r(G_2,o_2)\times K),\] where $\bar{\psi}(v,k)=(\psi(v),k)$. 

Two $h$-decorated \rwos\  $(G_1,o_1,T_1,A^{(1)}_1,\dots,A^{(h)}_1)$ and $(G_2,o_2,T_2,A^{(1)}_2,\dots,A^{(h)}_2)$ are called isomorphic if their distance is $0$. Let $\mathcal{\rwo}_h$ be the set of isomorphism classes of $h$-decorated \rwos. We also define $\mathcal{\rwo}_h(B)$ and $\mathcal{\wpc}_h$ the same way as their non-decorated versions were defined. With the above defined distance  they  are compact metric spaces. Similarly as before,  $\mathcal{P}(\mathcal{\rwo}_h(B))$ and $\mathcal{P}(\mathcal{\wpc}_h)$, endowed with the weak* topology, are compact spaces. Whenever the value of $h$ is clear from the context, we  omit it and simply use   the term "decorated \rwo".

A \emph{finite graph-positive-contraction} is a pair $(G,T)$, where $G$ is finite graph with degrees at most $D$, and $T$ is a positive contraction on $\ell^2(V(G)\times K)$. It can be turned into a random \wpc\ 
 \[U(G,T)=(G_o,o,T_o)\] by choosing $o$ as a uniform random vertex of $G$.

Note that all the definitions above depend on the choice of the finite set $K$. In most of the paper we can keep $K$ as fixed. Whenever we need to emphasize the specific choice of $K$, we will refer to $K$ as the support set of \rwos. Unless stated otherwise the support set is always assumed to be $K$. Let $L\subset K$ and let  $(G,o,T)$ be a \rwo\ with support set $K$. Let $P_L$ be the orthogonal projection from $\ell^2(V(G)\times K)$ to $\ell^2(V(G)\times L)\subset \ell^2(V(G)\times K)$. We define the operator $\rest_L(T)$ on $\ell^2(V(G)\times L)$ as $\rest_L(T)=P_L T\restriction_{\ell^2(V(G)\times L)}$. So $(G,o,\rest_L (T))$ is an \rwo\ with support set $L$.  

Sometimes we need to consider more than one operator on a rooted graph. A \emph{double \rwo} will mean a tuple $(G,o,T_1,T_2)$ where $(G,o)$ is a rooted graph and $T_1$, $T_2$ are bounded operators on $\ell^2(V(G)\times K)$. We omit the details how the set of isomorphism classes of double \rwos\ can be turned into a metric space. It is also clear what we mean by a decorated double \rwo, or  a triple \rwo, or a double \wpc. 


\subsection{Determinantal processes}

Let $E$ be a countable set, and $T$ be a positive contraction of $\ell^2(E)$. Then there is a random subset $B^T$ of $E$ with the property that for each finite subset $F$ of $E$ we have
\[\mathbb{P}[F\subset B^T]=\det(\langle Tx, y\rangle)_{x,y\in F},\]
where we identify an element $x\in E$ with its characteristic vector in $\ell^2(E)$. The distribution of $B^T$ is uniquely determined by these constraints, and it is called the \emph{determinantal measure} corresponding to $T$ \cite{lydet}.

Using the definition of the random subset $B^T$, we can define a map \break $\tau:\mathcal{\wpc}\to \mathcal{P}(\mathcal{\wpc}_1)$ by $\tau(G,o,T)=(G,o,T,B^T)$. This induces a  map \break $\tau_*:\mathcal{P}(\mathcal{\wpc})\to\mathcal{P}(\mathcal{P}(\mathcal{\wpc}_1))$. Taking  expectation  we get the map \break $\mathbb{E}\tau_*:\mathcal{P}(\mathcal{\wpc})\to\mathcal{P}(\mathcal{\wpc}_1)$. So given a random \wpc\ $(G,o,T)$ the meaning of $(G,o,T,B^T)$ is ambiguous. Unless stated otherwise $(G,o,T,B^T)$ will mean a random decorated \wpc, i.e., its distribution is an element of $\mathcal{P}(\mathcal{\wpc}_1)$.
\begin{prop}\label{prop1}
The maps $\tau,\tau_*$ and $\mathbb{E}\tau_*$ are continuous.
\end{prop}

\subsection{Trace and spectral measure}

Given a random \rwo\ $(G,o,T)$ we define 
\[\Tr(G,o,T)=\mathbb{E}\sum_{k\in K} \langle T(o,k),(o,k) \rangle.\]
We extend the definition to the decorated case in the obvious way.

Given a random \wpc\ $(G,o,T)$ its \emph{spectral measure} is the unique measure $\mu=\mu_{(G,o,T)}$ on $[0,1]$ with the property that, for any integer $n\ge 0$ we have
\[\Tr(G,o,T^n)=\int_0^1 x^n d\mu.\] 
Note that $\mu([0,1])=|K|$. Also if $T$ is a projection with probability $1$, then we have \[\mu=\Tr(G,o,T) \delta_1+(|K|-\Tr(G,o,T))\delta_0.\] If $(G,T)$ is a finite graph-positive-contraction, then the spectral measure of $U(G,T)$ can be obtained as \[\frac{1}{|V(G)|}\sum_{i=1}^{|V(G)\times K|} \delta_{\lambda_i},\] where $\lambda_1,\lambda_2,\dots,\lambda_{|V(G)\times K|}$ are the eigenvalues of $T$ with multiplicity.

\subsection{An equivalent characterization of tightness}

We already defined the notion of tightness in the Introduction. Here we repeat the definition in a slightly more general setting. For a finite graph-positive-contraction $(G,T)$ we define the measure $\nu_{(G,T)}$ on $\mathbb{N}\cup \{\infty\}$ by setting
\[\nu_{(G,T)}(\{t\})=|V(G)|^{-1}\sum_{\substack{(v_1,k_1),(v_2,k_2)\in V(G)\times K\\ d_G(v_1,v_2)=t}} |\langle T(v_1,k_1) ,(v_2,k_2)\rangle|^2,\] 
for all $t\in \mathbb{N}\cup \{\infty\}$. A sequence $(G_n,T_n)$ of finite graph-positive-contractions is tight if the family of measures $\nu_{(G_n,T_n)}$ is tight, that is, for each $\varepsilon>0$ we have a finite $R$ such that 
\[\nu_{(G_n,T_n)}\left(\{R+1,R+2,\dots\}\cup\{\infty\}\right)<\varepsilon\]
for all $n$. The next lemma gives an equivalent characterization of tightness.
\begin{lemma}\label{tightl}
Let  $(G_n,P_n)$ be a Benjamini-Schramm convergent sequence of finite graph-positive-contractions with limit $(G,o,T)$. Assume that $P_1,P_2,\dots$ are orthogonal projections. Then the following are equivalent
\begin{enumerate}[i)]
\item The sequence $(G_n,P_n)$ is tight.

\item The limit $T$ is an orthogonal projection with probability $1$ and \break $\nu_{(G_n,P_n)}(\{\infty\})=0$ for every $n$.
\end{enumerate} 
\end{lemma}
\begin{proof}
i)$\Rightarrow$ ii):  Recall the following well-known result. 
\begin{prop}\label{ptight}
Let $E$ be a countable set, and let $T$ be a positive contraction on $\ell^2(E)$. Then for all $e\in E$ we have $\langle T^2 e,e \rangle\le \langle T e,e \rangle$. Moreover, if for all $e\in E$ we have $\langle T^2 e,e \rangle= \langle T e,e \rangle$, then $T$ is an orthogonal projection.
\end{prop}

Let $(H_n,o_n,T_n)=U(G_n,P_n)$. Then
\begin{align*}
\nu_{(G_n,P_n)}(\mathbb{N}\cup\{\infty\})=|V(G_n)|^{-1} \Tr(P_n^* P_n)=|V(G_n)|^{-1} \Tr( P_n)=\Tr(H_n,o_n,T_n).
\end{align*}
 Combining this with the definition of tightness we get that for any $\varepsilon>0$ we have an $R$ such that 
\begin{equation} \label{ineqtight}
\mathbb{E} \sum_{k\in K}\sum_{(v,k')\in B_R(H_n,o_n)\times K } |\langle T_n (o_n,k),(v,k')\rangle|^2>\Tr(H_n,o_n,T_n)-\varepsilon
\end{equation}
for every $n$. 

Using the convergence of $(H_n,o_n,T_n)$ we get that 
\[\lim_{n\to\infty} \Tr(H_n,o_n,T_n)=\Tr(G,o,T),\]
and 
\begin{multline*}
\lim_{n\to\infty} \mathbb{E} \sum_{k\in K}  \sum_{(v,k')\in B_R(H_n,o_n) \times K}|\langle T_n (o_n,k),(v,k')\rangle|^2\\=\mathbb{E}\sum_{k\in K} \sum_{(v,k')\in B_R(G,o) \times K} |\langle T (o,k),(v,k')\rangle|^2.
\end{multline*}
Combining these with inequality (\ref{ineqtight}) we get that 
\begin{align*}
\Tr(G,o,T^2)&=\mathbb{E} \sum_{k\in K}\sum_{(v,k')\in V(G)\times K} |\langle T (o,k), (v,k')\rangle|^2\\&\ge 
\mathbb{E} \sum_{k\in K}\sum_{(v,k')\in B_R(G,o)\times K } |\langle T (o,k),(v,k')\rangle|^2 \ge \Tr(G,o,T)-\varepsilon.
\end{align*}
Tending to $0$ with $\varepsilon$ we get that
\[\Tr(G,o,T^2)\ge \Tr(G,o,T).\]
Combining this with the first statement of Proposition \ref{ptight} we get that with probability $1$ we have $\langle T^2(o,k),(o,k)\rangle=\langle T(o,k),(o,k)\rangle$ for every $k\in K$. But then it follows from the unimodularity of $(G,o,T)$ that with probability $1$ we have $\langle T^2(v,k),(v,k)\rangle=\langle T(v,k),(v,k)\rangle$ for any $(v,k)\in V(G)\times K$. See \cite[Lemma 2.3 (Everything Shows at the Root)]{ally} and Section~\ref{sec3}. Then Proposition \ref{ptight} gives us that $T$ is a projection with probability $1$. From the definition of tightness it is clear that $\nu_n(\{\infty\})=0$ for every $n$.

ii)$\Rightarrow$ i): Pick any $\varepsilon>0$. From the monotone convergence theorem and the fact that $T$ is a projection with probability $1$ we have

\begin{align*}
\Tr(G,o,T)&=\Tr(G,o,T^2)=\mathbb{E}\sum_{k\in K} \sum_{(v,k')\in V(G)\times K}  |\langle T (o,k), (v,k')\rangle|^2\\&=
\lim_{R\to\infty} \mathbb{E} \sum_{k\in K}\sum_{(v,k')\in B_R(G,o)\times K } |\langle T (o,k),(v,k')\rangle|^2 
\end{align*}

Thus, if we choose a large enough $R_0$, then we have
\[\Tr(G,o,T)-\mathbb{E} \sum_{k\in K}\sum_{(v,k')\in B_{R_0}(G,o)\times K } |\langle T (o,k),(v,k')\rangle|^2 < \frac{\varepsilon}{2}.\]
Then from the convergence of $(H_n,o_n,T_n)$ we get that there is an $N$ such that if $n>N$ we have
\begin{multline*}
\nu_{(G_n,P_n)}(\{R_0+1,R_0+2,\dots\}\cup\{\infty\})=\\ \Tr(H_n,o_n,T_n)-\mathbb{E} \sum_{k\in K}\sum_{(v,k')\in B_{R_0}(H_n,o_n)\times K } |\langle T_n (o_n,k),(v,k')\rangle|^2 <\varepsilon.
\end{multline*}
Using the condition that $\nu_{(G_n,P_n)}(\{\infty\})=0$ for all $n$ and the definition of $\nu_{(G_n,P_n)}$ we get that the support of the  measure $\nu_{(G_n,P_n)}$ is contained in $\{0,1,\dots,|V(G_n)|\}$. 
Thus, the choice \break $R=\max(R_0,|V(G_1)|,|V(G_2)|,\dots,|V(G_N)|)$ is good for  $\varepsilon$.
\end{proof}

\subsection{Sofic entropy}

\if\sver 1
 Let $C$ be a finite set and let $\Gamma$ be a finitely generated group. Let $f$ be a random coloring of $\Gamma$ with $C$, that is a random element of $C^{\Gamma}$. (The measurable structure of $C^{\Gamma}$ comes from the product topology on $C^{\Gamma}$.) Given a coloring $f\in C^{\Gamma}$ and $\gamma\in \Gamma$ we define the coloring $f_\gamma$ by $f_\gamma(g)=f(\gamma^{-1}g)$ for all $g\in \Gamma$. This notation extends to random colorings in the obvious way. A random coloring $f$ is invariant if for every $\gamma\in \Gamma$ the distribution of $f_\gamma$ is the same as the distribution of $f$.

Now assume that $\Gamma$ is a finitely generated sofic group, and $f$ is an invariant random coloring of $\Gamma$. Let $S$ be a finite generating set, and  let $G_1,G_2,\dots$ be a sequence of $S$-labeled Schreier-graphs Benjamini-Schramm converging to the Cayley-graph $G_\Gamma=\text{Cay}(\Gamma,S)$. Now we define the so called \emph{sofic entropy} of $f$. There are many slightly different versions of this notion \cite{sofic1,sofic2}, we will follow Ab\'{e}rt and Weiss \cite{abwe}. Let $G$ be  a finite $S$-labeled Schreier graph and $g$ be a random coloring  of $V(G)$. Given  $\varepsilon>0$ and a positive integer $r$,   we say that $g$ is an $(\varepsilon,r)$ approximation of $f$ on the graph $G$, if there are at least $(1-\varepsilon)|V(G)|$ vertices $v\in V(G)$, such that $B_r(G,v)$ is isomorphic to $B_r(G_\Gamma,e_\Gamma)$, moreover  $d_{TV}(f\restriction B_r(G_\Gamma,e_\Gamma),g\restriction B_r(G,v))<\varepsilon$, where $d_{TV}$ is the total variational distance, and it is meant that we identify $B_r(G_\Gamma,e_\Gamma)$ and $B_r(G,v)$. Let us define 
 \[H(G,\varepsilon,r)=\sup\left\{\frac{H(g)}{|V(G)|}\quad \Big|\quad g\text{ is an $(\varepsilon,r)$ approximation of $f$ on $G$}\right\}.\] Here $H(g)$ is the Shannon-entropy of $g$. Let $H(\varepsilon,r)$ be the supremum of $H(G,\varepsilon,r)$, over all finite $S$-labeled Schreier graphs $G$. We define two versions of sofic entropy. The first one    
 \[h(f)=\inf_{\varepsilon,r}\limsup_{n\to\infty} H(G_n,\varepsilon,r).\]
 Note that this might depend on the chosen sofic approximation. Another option is to define sofic entropy as
 \[h'(f)=\inf_{\varepsilon,r} H(\varepsilon,r).\] 
 Observe that $h'(f)\ge h(f)$. It is open whether $h'(f)= h(f)$ for any sofic approximation apart from trivial counterexamples. We can also express these quantities as
 \[h(f)=\inf_{\varepsilon}\limsup_{n\to\infty} H(G_n,\varepsilon,\lfloor \varepsilon^{-1}\rfloor)\text{ and }h'(f)=\inf_{\varepsilon} H(\varepsilon,\lfloor \varepsilon^{-1}\rfloor).\]

 The quantities $h(f)$ and $h'(f)$ are isomorphism invariants in the abstract ergodic theoretic sense.

\else
 Let $Q$ be a vertex transitive connected infinite graph. For example, let $\Gamma$ be an infinite group with finite generating set $S$, and let $Q$ be the Cayley graph $\text{Cay}(\Gamma,S)$. In this case the automorphisms of $Q=\text{Cay}(\Gamma,S)$ are  required to respect the orientation and labeling of the edges, which implies $\Aut(Q)\cong \Gamma$.  Let $C$ be a finite set, and let $f$ be a random coloring of $V(Q)$ with $C$, that is a random element of $C^{V(Q)}$. (The measurable structure of $C^{V(Q)}$ comes from the product topology on $C^{V(Q)}$. ) Given a coloring $f\in C^{V(Q)}$ and $\gamma\in \Aut(Q)$ we define the coloring $f_\gamma$ by $f_\gamma(v)=f(\gamma^{-1}v)$ for all $v\in V(Q)$. This notation extends to random colorings in the obvious way. A random coloring $f$ is invariant if for every $\gamma\in \Aut(Q)$ the distribution of $f_\gamma$ is the same as the distribution of $f$.

Let $o_Q$ be any vertex of $Q$. Assume that $(Q,o_Q)$ is a Benjamini-Schramm limit of finite graphs. Note that if $Q=\text{Cay}(\Gamma,S)$, then this means that $\Gamma$ is a sofic group.  Let $f$ is an invariant random coloring of $\Gamma$.   Let $G_1,G_2,\dots$ be a sequence of finite  graphs Benjamini-Schramm converging to $(Q,o_Q)$. Now we define the so called sofic entropy of $f$. There are many slightly different versions of this notion \cite{sofic1,sofic2}, we will follow Ab\'{e}rt and Weiss \cite{abwe}. Let $G$ be  a finite  graph, $g$ be a random coloring  of $V(G)$. Given  $\varepsilon>0$ and a positive integer $r$,   we say that $g$ is an $(\varepsilon,r)$ approximation of $f$ on the graph $G$, if there are at least $(1-\varepsilon)|V(G)|$ vertices $v\in V(G)$, such that $B_r(G,v)$ is isomorphic to $B_r(Q,o_Q)$, moreover  $d_{TV}(f\restriction B_r(Q,o_Q),g\restriction B_r(G,v))<\varepsilon$, where $d_{TV}$ is the total variational distance, and it is meant that we identify $B_r(Q,o_Q)$ and $B_r(G,v)$. Let us define 
\[H(G,\varepsilon,r)=\sup\left\{\frac{H(g)}{|V(G)|}\quad \Big|\quad g\text{ is an $(\varepsilon,r)$ approximation of $f$ on $G$}\right\}.\] Here $H(g)$ is the Shannon entropy of $g$. Let $H(\varepsilon,r)$ be the supremum of $H(G,\varepsilon,r)$, over all finite graphs $G$. We define two versions of sofic entropy. The first one    
\[h(f)=\inf_{\varepsilon,r}\limsup_{n\to\infty} H(G_n,\varepsilon,r).\]
Note, that this might depend on the chosen  approximation. An other option is to define sofic entropy as
\[h'(f)=\inf_{\varepsilon,r} H(\varepsilon,r).\] 
Note that $h'(f)\ge h(f)$. It is open whether $h'(f)= h(f)$ for any sofic approximation apart from trivial counter examples. Note that we can also express these quantities as
\[h(f)=\inf_{\varepsilon}\limsup_{n\to\infty} H(G_n,\varepsilon,\lfloor \varepsilon^{-1}\rfloor)\text{ and }h'(f)=\inf_{\varepsilon} H(\varepsilon,\lfloor \varepsilon^{-1}\rfloor).\]

If $Q=\text{Cay}(\Gamma,S)$, the quantities $h(f)$ and $h'(f)$ are isomorphism invariants in the abstract ergodic theoretic sense.
\fi

\if\sver 1
\emph{Remark.} 
Sofic entropy can be defined in a more general setting. Namely, let $Q$ be a locally finite vertex transitive graph.  Let $o$ be any vertex of it. Assume that $(Q,o)$ is a Benjamini-Schramm limit of finite graphs. Let $f$ be a random coloring of $V(Q)$ with $C$ such that the distribution of $f$ is invariant under all automorphisms of $Q$. We would like to define the sofic entropy of $f$ the same way as above. The only problematic point is that in the definition of $(\varepsilon,r)$-approximation we need to identify $B_r(G,v)$ with $B_r(Q,o)$. But $B_r(Q,o)$ might have non-trivial automorphisms, in which case there are more than one possible identifications and it is not clear which  we should choose. If all the automorphisms $B_r(Q,o)$ can be extended to an automorphism of $Q$, then we can choose any identification, because they all give the same total variation distance. But if $B_r(Q,o)$ has other automorphisms then things get more complicated. However, one can overcome these difficulties and get a sensible notion of sofic entropy \cite{abwe}.  Here we do not give the details, we just mention that Theorem \ref{SoficThm} stated in the next subsection can be extended to this more general setting. 


\fi
 
\subsection{Our main theorems}\label{subsecmain}

Let $E$ be a countable set, and $T$ be a positive contraction on  $\ell^2(E)$. Let $c$ be a $[0,1]$ labeling of $E$. For $e\in E$ let $I(e)$ be the indicator of the event that $e\in B^T$.   
 For $e\in E$ we define 
\[\bar{h}(e,c,T)=H(I(e)|\{I(f)|c(f)<c(e)\}).\]
Here, $H$ is the conditional entropy, that is, with the notation \[g(x)=-x\log x-(1-x)\log(1-x),\] we have
\[H(I(e)|\{I(f)|c(f)<c(e)\})=\mathbb{E}g(\mathbb{E}[I(e)|\{I(f)|c(f)<c(e)\}]).\]
Moreover, we define
\[\bar{h}(e,T)=\mathbb{E}\bar{h}(e,c,T),\]
where $c$ is an i.i.d. uniform $[0,1]$ labeling of $E$. 

For a random \wpc\ $(G,o,T)$  we define
\[\bar{h}(G,o,T)=\mathbb{E}\sum_{k\in K} \bar{h}((o,k),T).\]
 



If $L\subset K$ and $(G,T)$ is a finite graph-positive-contraction we define  $h_L(G,T)$ to be the Shannon entropy of $B^T\cap(V(G)\times L)$. 

\begin{theorem}\label{MainThm}
Let $(G_n,P_n)$ be a sequence of finite graph-positive-contractions, such that  
$\lim_{n\to\infty} U(G_n,P_n)=(G,o,P)$ for some random \wpc\ $(G,o,P)$. Assume that $P_1,P_2,\dots$ are orthogonal projections, and $P$ is an orthogonal projection with probability $1$. Let $L\subset K$. Then 
\[\lim_{n\to\infty} \frac{h_L(G_n,P_n)}{|V(G_n)|}=\bar{h}(G,o,\rest_L(P)).\] 
\end{theorem}

Using Lemma \ref{tightl} we immediately get the following theorem. 

\begin{theorem}\label{MainThm_t}
Let $(G_n,P_n)$ be a \textbf{tight} sequence of finite graph-positive-contractions, such that  
$\lim_{n\to\infty} U(G_n,P_n)=(G,o,P)$ for some random \wpc\ $(G,o,P)$. Assume that $P_1,P_2,\dots$ are orthogonal projections. Let $L\subset K$. Then 
\[\lim_{n\to\infty} \frac{h_L(G_n,P_n)}{|V(G_n)|}=\bar{h}(G,o,\rest_L(P)).\] 
\end{theorem}

\if \sver 0

Let $Q$ be a vertex transitive connected infinite graph which is a Benjamini-Schramm limit of finite graphs. A positive contraction $T$ on $\ell^2(V(Q)\times K)$ is called invariant, if for any $\gamma\in\Aut(Q)$ and $v_1,v_2\in \Gamma$ and $k_1,k_2\in K$ we have
\[\langle T(v_1,k_1),(v_2,k_2)\rangle=\langle T(\gamma^{-1}v_1,k_1),(\gamma^{-1}v_2,k_2)\rangle.\]

For an invariant positive contraction, if we regard the random subset $B^T$ as a random coloring with $\{0,1\}^K$, we see that $B^T$ is an invariant coloring. Thus we can speak about its sofic entropy. 

%


%

\begin{theorem}\label{SoficThm}
Let $Q$ be a vertex transitive connected infinite graph which is a Benjamini-Schramm limit of finite graphs. If $T$ is an invariant positive contraction on $\ell^2(V(Q)\times K)$ then we have
\[h(B^T)=h'(B^T)= \bar{h}(Q,o_Q,T)\]
for any sofic approximation of $Q$.
\end{theorem}

\else

Let $\Gamma$ be a finitely generated sofic group. A positive contraction $T$ on $\ell^2(\Gamma\times K)$ is called invariant, if for any $\gamma,g_1,g_2\in \Gamma$  and $k_1,k_2\in K$ we have
\[\langle T(g_1,g_1),(g_2,k_2)\rangle=\langle T(\gamma^{-1}g_1,k_1),(\gamma^{-1}g_2,k_2)\rangle.\]

For an invariant positive contraction if we regard the random subset $B^T$ as a random coloring with $\{0,1\}^K$, we see that $B^T$ is an invariant coloring. Thus we can speak about its sofic entropy. 

%


%
As before let $S$ be a finite generating set of $\Gamma$, let $e_\Gamma$ be the identity  of $\Gamma$, and  $G_\Gamma=\text{Cay}(\Gamma,S)$ be the Cayley-graph of $\Gamma$.

\begin{theorem}\label{SoficThm}
Let $\Gamma$ be a finitely generated sofic group. If $T$ is an invariant positive contraction
on $\ell^2(\Gamma\times K)$ then we have
\[h(B^T)=h'(B^T)= \bar{h}(G_\Gamma,e_\Gamma,T)\]
for any sofic approximation of $\Gamma$.
\end{theorem}

\fi

Note that we can easily generalize the definition of $\bar{h}$ to any invariant random coloring $f$. It is known that even in this more general setting $\bar{h}$ is an upper bound on the sofic entropy. However, $\bar{h}$ is not an isomorphism invariant in the ergodic theoretic sense. See \cite{seward}. 

The random ordering idea above was used by  Borgs,  Chayes,  Kahn and Lov\'asz \cite{borgs} to give the growth of the partition function and entropy of certain Gibbs measures  at high temperature on Benjamini-Schramm convergent graph sequences. See also \cite{aumo}.   

\subsection{An example: Why tightness is necessary?}

We  consider two connected graphs $H_1$ and $H_2$. Let $H_1$ be the complete graph on $4$ vertices, and let $H_2$ be the graph that is  obtained from a star with $3$ edges by doubling each edge. Both have $4$ vertices and $6$ edges.   Let $T_i$ be a uniform random spanning tree of $H_i$, and let $P_i$ be the corresponding $6\times 6$ transfer-current matrix. It is straightforward to check that for any $e\in E(H_i)$ we have $\mathbb{P}(e\in T_i)=\frac{1}{2}$. Thus, in both $P_1$ and $P_2$ all the diagonal entries are equal to $\frac{1}{2}$. Now let $G_i$ be the empty graph on the vertex set $E(H_i)$. Then the pairs $(G_1,P_1)$ and $(G_2,P_2)$ are indistinguishable by local sampling, that is, $U(G_1,P_1)$ and $U(G_2,P_2)$ have the same distribution. On the other hand $H_1$ has $16$ spanning trees, and $H_2$ has only $8$ spanning trees. So $|V(G_1)|^{-1} H(B^{P_1})\neq |V(G_2)|^{-1} H(B^{P_2})$. This shows that the condition of tightness can not be omitted  in Theorem~\ref{MainThm_t}. One could think that this only works, because the graphs $G_1$ and $G_2$ are not connected. But Theorem \ref{MainThm_t} still fails without the assumption of tightness, even if we assume that all the graphs are connected.  We sketch the main idea. Let $i\in \{1,2\}$. For each $n$ we consider a block diagonal matrix $B_{i,n}$, where we have $n$ diagonal blocks each of which equal to $P_i$. Then we take a connected graph $G_{i,n}$ on $V_{i,n}$ (the set of columns of $B_{i,n}$) in such a way that if two columns are in the same block, then they must be at least at distance $d(n)$ in the graph $G_{i,n}$ for some $d(n)$ tending to infinity. Moreover, we can choose $G_{i,n}$ such that the sequences $(G_{1,n})$ and $(G_{2,n})$ have the same Benjamini-Schramm limit $(G,o)$. Then both of the sequences $(G_{1,n},B_{1,n})$ and $(G_{2,n},B_{2,n})$  have the same limit, namely, $(G,o,\frac{1}{2}I)$. But their asymptotic entropy is different.

\section{Unimodularity and conditional determinantal processes} \label{sec3}

\subsection{Unimodularity}

We define bi-rooted graph-operators as tuples  $(G,o,o',T)$, where $G$ is a connected graph with degree bound $D$, $o,o'\in V(G)$ and $T$ is a bounded operator on $\ell^2(V(G)\times K)$. Let $bi\mathcal{\rwo}$ be the set of isomorphism classes of bi-rooted graph-operators. We omit the details how to endow this space with a measurable structure. A random \rwo\  $(G,o,T)$ is called \emph{unimodular}, if for any non-negative measurable function $f:bi\mathcal{\rwo}\to\mathbb{R}$ we have
\[\mathbb{E}\sum_{v\in V(G)} f(G,o,v,T)=\mathbb{E}\sum_{v\in V(G)} f(G,v,o,T).\]

The next lemma gives some examples  of unimodular random \rwos. The proof goes like the one given in \cite{besch}.
\begin{lemma}
If $(G,T)$ is a  finite graph-positive-contraction, then $U(G,T)$ is unimodular. The limit of unimodular random \rwos\ is unimodular.  
\end{lemma} 

Of course the notion of unimodularity can be extended to double/triple (decorated) \rwos. We will use the following consequence of unimodularity.

\begin{lemma}\label{lemmatraceseg}
Let $(G,o,T,S)$ be a unimodular random double \rwo. Assume that there is a finite $B$ such that $\|T\|,\|S\|<B$ with probability $1$. Then
\[\Tr(G,o,TS)=\Tr(G,o,ST).\]
\end{lemma}
\begin{proof}
The proof is the same as in \cite[Section 5]{ally}.
\end{proof}

It has the following consequences.
\begin{lemma}\label{lranknullity}
In the following statements we always assume that $P$ and $P_i$ are all orthogonal projections with probability $1$. 

\begin{enumerate}
\item\label{elso} Let $(G,o,P_1,P_2,U)$ be a unimodular random triple \rwo, such that with probability $1$ we have $U\restriction \ker P_1\equiv 0$ and $U\restriction \range P_1$ is an isomorphism between $\range P_1$ and $\range P_2$. Then
\[\Tr(G,o,P_1)=\Tr(G,o,P_2).\] 
\item\label{masodik} Let $(G,o,P_1,P_2,T)$ be a unimodular random triple \rwo, such that with probability $1$ we have $\overline{\range TP_1}=\range P_2$ and $T$ is injective on  $\range P_1$. Then
\[\Tr(G,o,P_1)=\Tr(G,o,P_2).\]
\item \label{ranknullity}(rank-nullity theorem) Let $(G,o,P,P_1,P_2,T)$ be a unimodular random quadruple  \rwo, such that with probability $1$ we have that $P_1$ is the orthogonal projection to $\ker (T\restriction \range P)$ and $P_2$ is the orthogonal projection to $\overline{\range (T\restriction \range P)}$. Then
\[\Tr(G,o,P)=\Tr(G,o,P_1)+\Tr(G,o,P_2).\] 
\end{enumerate}
\end{lemma}
\begin{proof}
To prove part \ref{elso} observe that $P_1U^*U=P_1$ and $UP_1U^*=P_2$. Note that all operators have norm at most $1$, so from Lemma \ref{lemmatraceseg}
\[\Tr(G,o,P_1)=\Tr(G,o,(P_1U^*)U)=\Tr(G,o,U(P_1U^*))=\Tr(G,o,P_2).\]

To prove part \ref{masodik} let $TP_1=UH$ be the unique polar decomposition of $TP_1$, then $(G,o,P_1,P_2,UP_1)$ satisfies the conditions in part \ref{elso}, so the statement follows. The rather technical details why the polar decomposition is measurable are given in the Appendix. Note that once we established the measurability of $U$,  unimodularity follows from the uniqueness of the decomposition.

To prove part \ref{ranknullity} let $H=\range P\cap (\ker T\restriction \range P)^{\bot}$. Let $P_H$ be the orthogonal projection to $H$, then we have $P=P_1+P_H$. Therefore, $\Tr(G,o,P)=\Tr(G,o,P_1)+\Tr(G,o,P_H)$. It is also clear that  $\overline{\range TP}=\overline{\range (T\restriction H)}$ and $T$ is injective on $H$. Thus part \ref{masodik} gives us $\Tr(G,o,P_H)=\Tr(G,o,P_2)$. Putting everything together we obtain that
\[\Tr(G,o,P)=\Tr(G,o,P_1)+\Tr(G,o,P_H)=\Tr(G,o,P_1)+\Tr(G,o,P_2).\]
\end{proof}

\subsection{Conditional determinantal processes}
Let $P$ be an orthogonal projection to a closed subspace $H$ of $\ell^2(E)$. Given $C\subset E$, let $[C]$ be the closed subspace generated by $e\in C$, and let $[C]^{\bot}$ be the orthogonal complement of it. 
Note that $[C]^{\bot}=[E\backslash C]$. We define $P_{/C}$ as the orthogonal projection to the closed subspace $(H\cap [C]^{\bot})+[C]$, and $P_{\times C}$ as the orthogonal projection to the closed subspace $H\cap [C]^{\bot}$. We also define $P_{-C}=I-(I-P)_{/C}$.

\begin{prop}\label{pertimes}
We have $P_{/C}=P_{\times C}+P_{[C]}$, where $P_{[C]}$ is the orthogonal projection to $[C]$. In other words $P_{/C}e=e$ for $e\in C$ and $P_{/C}e=P_{\times C} e$ for $e\in E\backslash C$.  Moreover, if $C_n$ is an increasing sequence of subsets of $E$ and $C=\cup C_n$, then $P_{/C_n}$ converges to $P_{/C}$ in the strong operator topology. Furthermore, the sequence $\langle P_{\times C_n}e,e\rangle$ is monotone decreasing. 
\end{prop}
\begin{proof}
The first statement is trivial. To prove the second statement, observe that $P_{\times C_n}$ is a sequence of orthogonal projections to a monotone decreasing sequence of closed subspaces with intersection $\range P_{\times C}$, so $P_{\times C_n}$ converge to $P_{\times C}$ in the strong operator topology. It is also clear that $P_{[C_n]}$ converge to $P_{[C]}$, so from $P_{/ C_n}=P_{\times C_n}+P_{[C_n]}$ the statement follows. To prove the third statement observe that $\langle P_{\times C_n} e,e \rangle=\|P_{\times C_n} e\|_2^2$. So the statement follows again from the fact that $P_{\times C_n}$ is a sequence of orthogonal projections to a monotone decreasing sequence of closed subspaces.  
\end{proof}
  For $C,D\subset E$ we define $P_{/C-D}=(P_{/C})_{-D}$, and we define $P_{-D/C}=(P_{-D})_{/C}$. We only include the next lemma here to make it easier to compare formulas in \cite{lydet} with our formulas.

\begin{lemma}
Let $P$ be an orthogonal projection to a closed subspace $H$. Then for any $D\subset E$ we have 
\[\range P_{-D}=\overline{H+[D]}\cap[D]^{\bot}.\]
Moreover, if $C$ and $D$ are disjoint subsets of $E$, then
\[\range P_{/C-D}=\overline{(H\cap [C]^{\bot})+[C\cup D] }\cap[D]^\bot\]
and
\[\range P_{-D/C}=(\overline{H+[D]}\cap [C\cup D]^\bot)+[C].\]
If $C$ and $D$ are finite, then the above formulas are true even if we omit the closures.
\end{lemma} 
\begin{proof}
We only prove the first statement. The other statements can be easily deduced from it. Unpacking the definitions we need to prove that
\[((H^{\bot}\cap [D]^{\bot})+[D])^{\bot}=\overline{H+[D]}\cap[D]^{\bot}.\]
As a first step observe that $\overline{H+[D]}\cap[D]^{\bot}=\overline{\range (P_{[D]^{\bot}}\restriction H)}$. Indeed, if $x\in \overline{(\range P_{[D]^{\bot}}\restriction H)}$, then $x=\lim x_n$, where for all $n$ we have $x_n\in [D]^{\bot}$ and there is an $y_n\in [D]$ such that \break $x_n+y_n\in H$. But then $x_n=(x_n+y_n)-y_n\in H+[D]$, which implies that $x\in \overline{H+[D]}$.  Clearly $x\in [D]^\bot$, so $x\in \overline{H+[D]}\cap[D]^{\bot}$.

 To prove the other containment let $x\in \overline{H+[D]}\cap[D]^{\bot}$, then $x=\lim x_n$ where \break $x_n=y_n+z_n$ with $y_n\in H$ and $z_n\in [D]$. Since $P_{[D]^{\bot}}$ is continuous, we have \[x=P_{[D]^\bot}x=\lim P_{[D]^{\bot}} (y_n+z_n)= \lim P_{[D]^{\bot}} y_n\in \overline{\range (P_{[D]^{\bot}}\restriction H)}.\] Now it is easy to see that we need to prove that
\[(H^{\bot}\cap [D]^{\bot})+[D]=(\range P_{[D]^{\bot}}\restriction H)^{\bot}.\]  
First let $x\in (\range P_{[D]^{\bot}}\restriction H)^{\bot}$. Then for any $h\in H$ we have
\[0=\langle x,P_{[D]^{\bot}} h\rangle=\langle P_{[D]^{\bot}} x,h\rangle,\]
which implies that $P_{[D]^{\bot}}x\in H^{\bot}\cap [D]^{\bot}$. Thus, ${x=P_{[D]^{\bot}}x+P_{[D]}x\in( H^{\bot}\cap [D]^{\bot})+[D]}$. To show the other containment let us consider $x=y+z$ such that $y\in H^{\bot}\cap [D]^{\bot}$ and $z\in [D]$. Then for any $h\in H$ we have
\[\langle x,P_{[D]^{\bot}} h\rangle=\langle P_{[D]^{\bot}}x, h\rangle=\langle y, h\rangle=0,\]
because $y\in H^{\bot}$.

For the last statement, see the discussion in the paper \cite{lydet} after the proof of Corollary~6.4.
\end{proof}

We have the following lemma. See \cite[Equation (6.5)]{lydet}.

\begin{lemma}\label{fincond}
Let $C$ and $D$ be disjoint finite subsets of $E$ such that \break $\mathbb{P}[B^P\cap(C\cup D)=C]>0$. Then $P_{/C-D}=P_{-D/C}$ and conditioned on the event \break $B^P\cap(C\cup D)=C$, the distribution of $B^P$ is the same as that of $B^{P_{/C-D}}$. 
\end{lemma}

The lemma above shows why the pairs $(C,D)$ of finite disjoint sets with the property that  $\mathbb{P}[B^P\cap(C\cup D)=C]>0$ are interesting for us. The next proposition gives an equivalent characterization of these pairs.
\begin{prop}
Let $C$ and $D$ be disjoint finite subsets of $E$. Then we have  \break $\mathbb{P}[B^P\cap(C\cup D)=C]>0$ if and only if $\range P_{[C]}P=[C]$ and $\range P_{[D]} (I-P)=[D]$. 
\end{prop} 

This motivates the following definitions. A (not necessary finite) subset  $C$ of $E$ is called \emph{independent} (with respect to $P$) if       $\overline{\range P_{[C]}P}=[C]$. A  subset  $D$ of $E$ is called \emph{dually} independent (with respect to $P$) if $\overline{\range P_{[D]}(I-P)}=[D]$. A pair $(C,D)$ of subsets of $E$ is called \emph{permitted} (with respect to $P$) if $C$ and $D$ are disjoint, $C$ is independent and $D$ is dually independent.

We will need the following theorem of Lyons \cite[Theorem 7.2]{lydet}.
\begin{theorem}\label{permthm}
The pair $(B^{P},E\backslash B^{P})$ is permitted with probability $1$.
\end{theorem}    

We will also need the following statements.

\begin{prop}\label{permprop}
If $(C,D)$ is permitted, $C'\subset C$ and $D'\subset D$, then $(C',D')$ is permitted. 
\end{prop}

\begin{prop}\label{propfele}
Assume $(C,D)$ is a permitted pair. Then $D$ is dually independent with respect to $P_{/C}$, or equivalently, $D$ is independent with respect to $I-P_{/C}$.
\end{prop}
\begin{proof}
By the definition of a permitted pair $\overline{\range P_{[D]}(I-P)}=[D]$, so it is enough to show that $\range P_{[D]}(I-P)\subset \range P_{[D]} (I-P_{/C})$. 
Take any  $r\in \range P_{[D]}(I-P)$, then there is $x$ such that $r=P_{[D]}(I-P)x$. Let $y=P_{[C]^{\bot}}(I-P)x$. 
We claim that $y\in \range(I-P_{/C})$, or in other words, $y$ is orthogonal to any element $w\in \range P_{/C}$. We can write $w$ as $w=w_0+w_1$, where $w_0\in \range P\cap [C]^\bot$ and $w_1\in [C]$. We have 
\[\langle y,w_0\rangle=\langle P_{[C]^{\bot}}(I-P)x,w_0\rangle=\langle (I-P)x,P_{[C]^{\bot}}w_0\rangle=\langle (I-P)x ,w_0\rangle=0, \]  
since $w_0\in \range P$. Moreover $\langle y,w_1\rangle=0$, because $y\in [C]^\bot$ and $w_1\in [C]$. Thus, $\langle y,w\rangle=0$, so $y$ is indeed in the image of $I-P_{/C}$, then $P_{[D]}y$ is in the image of $P_{[D]} (I-P_{/C})$. Using that $C$ and $D$ are disjoint $P_{[D]}y=P_{[D]}P_{[C]^{\bot}}(I-P)x=P_{[D]}(I-P)x=r$. 
\end{proof}

Assume for a moment that $E$ is finite, then $|B^P|=\dim\range P$ with probability $1$. If $(C,D)$ is a permitted pair, then the distribution of $B^{P_{/C-D}}$ is the same as that of  $B^{P}$ conditioned on the event that $B^{P}\cap(C\cup D)=C$. So $|B^{P_{/C-D}}|=\dim\range P$ with probability $1$. In particular, $\mathbb{E}|B^P|=\mathbb{E}|B^{P_{/C-D}}|$. The next lemma extends this statement to the more general unimodular setting.

\begin{lemma}\label{lemmaTr}
Let $(G,o,P,C,D)$ be a unimodular random decorated \wpc\, where $P$ is an orthogonal projection and the pair $(C,D)$ is permitted with probability $1$. Then
\[\Tr(G,o,P)=\Tr(G,o,P_{/C-D})=\Tr(G,o,P_{-D/C}).\]
\end{lemma}
This can be obtained from combining Proposition \ref{propfele} and the following lemma.

\begin{lemma}\label{lemmaTr2}
Let $(G,o,P,C)$ be a unimodular random decorated \wpc\, where $P$ is an orthogonal projection and $C$ is independent with probability $1$. Then
\[\Tr(G,o,P)=\Tr(G,o,P_{/C}).\] 
We also have the corresponding dual statement, that is, let $(G,o,P,D)$ be a unimodular random decorated \wpc\, where $P$ is an orthogonal projection and $D$ is dually independent with probability $1$. Then
\[\Tr(G,o,P)=\Tr(G,o,P_{-D}).\] 
\end{lemma}

\begin{proof}
We only need to prove the first statement, because the second one can be obtained by applying the first statement to $I-P$.

Observe that $\ker (P_{[C]}\restriction \range P)=\range P_{\times C}$ from the definition of $P_{\times C}$, moreover, \break $\overline{\range  (P_{[C]}\restriction \range P)}=[C]$, because $C$ is independent. Applying the rank nullity theorem (Lemma \ref{lranknullity}.\ref{ranknullity}) and then using the fact $P_{/C}=P_{\times C}+P_{[C]}$ from Proposition \ref{pertimes} we get that
\[
\Tr(G,o,P)=\Tr(G,o,P_{\times C})+\Tr(G,o,P_{[C]})=\Tr(G,o,P_{\times C}+P_{[C]})=\Tr(G,o,P_{/C}).
\]
\end{proof}

The next lemma gives an extension of Lemma \ref{fincond}.

\begin{lemma}\label{lemmakond1}
Let $F\subset E$, and assume that \[\langle P_{/B^P\cap F-F\backslash B^P}e,e\rangle=\langle P_{-F\backslash B^p /B^P\cap F}e,e\rangle\] for all $e\in E$ with probability $1$. Then for any finite  $A\subset E$ we have

\[\mathbb{P}(A\subset B^P|B^P\restriction F)=\mathbb{P}(A\subset B^{P_{/B^P\cap F-F\backslash B^P}}).\]
\end{lemma}
\begin{proof}
Let $F_1,F_2,\dots$ be an increasing sequence of finite sets such that their union is $F$. The  crucial step in the proof is the following lemma.  

\begin{lemma}\label{lemmakond2}
Let $(C,D)$ be a permitted pair, such that $C\cup D=F$. Then $\langle P_{/C-D}e,e\rangle\le \langle P_{-D/C}e,e\rangle$ for all $e\in E$. Now assume that $\langle P_{/C-D}e,e\rangle=\langle P_{-D/C}e,e\rangle$ for all $e\in E$. Let us define $P_n=P_{/C\cap F_n-D\cap F_n}$. Then $B^{P_{/C-D}}$ is the weak limit of $B^{P_n}$. 
\end{lemma}
\begin{proof}
Let $A$ be a finite set such that, $A\cap F=\emptyset$, moreover let $\mathcal{A}$ be an upwardly closed subset of $2^A$, i.e. if $X\subset Y\subset A$ and $X\in\mathcal{A}$, then $Y\in\mathcal{A}$.   Using that determinantal measures have negative  associations
(\cite[Theorem 6.5]{lydet}) we get the following inequality for $m>n$ 
\[\mathbb{P}[ B^{P_n}\cap A\in \mathcal{A}]=\mathbb{P}[B^{P_{/C\cap F_n-D\cap F_n}}\cap A\in \mathcal{A}]\ge \mathbb{P}[B^{P_{/C\cap F_m-D\cap F_n}}\cap A\in \mathcal{A}] .\]
Tending to infinity with $m$, we get that
\begin{equation}\label{inq1}
\mathbb{P}[B^{P_n}\cap A\in \mathcal{A}]\ge \mathbb{P}[B^{P_{/C-D\cap F_n}}\cap A\in \mathcal{A}].
\end{equation}
To justify this last statement, let $\mathcal{U}$ be the set of orthogonal projections $R$ such that $D\cap F_n$ is dually independent with respect to $R$. Combining Proposition~\ref{permprop} and Proposition~\ref{propfele}, we obtain that $P_{/C\cap F_m}$ and $P_{/C}$ are all contained in $\mathcal{U}$.  For $R\in \mathcal{U}$, the probability  \break $\mathbb{P}[ B^{R_{-D\cap F_n}}\cap A\in \mathcal{A}]$ is a continuous function of $(\langle Re,f\rangle)_{e,f\in A\cup(D\cap F_n)}$. As  we proved in Proposition \ref{pertimes},  $P_{/C\cap F_m}$ tends to $P_{/C}$ in the strong operator topology.  
Thus,  \[\lim_{m\to\infty}\mathbb{P}[B^{P_{/C\cap F_m-D\cap F_n}}\cap A\in \mathcal{A}]=\mathbb{P}[B^{P_{/C-D\cap F_n}}\cap A\in \mathcal{A}].\]
This gives us Inequality \eqref{inq1}.

Tending to infinity with $n$ we get that
\[\liminf_{n\to\infty} \mathbb{P}[ B^{P_n}\cap A\in \mathcal{A}]\ge \lim_{n\to\infty} \mathbb{P}[B^{P_{/C-D\cap F_n}}\cap A\in \mathcal{A}]=\mathbb{P}[ B^{P_{/C-D}}\cap A\in \mathcal{A}]. \]
A similar argument gives that
\[\limsup_{n\to\infty} \mathbb{P}[B^{P_n}\cap A\in \mathcal{A}]\le \mathbb{P}[B^{P_{-D/C}}\cap A\in \mathcal{A}]\]
Therefore,
\begin{align}\label{ineqst}
\mathbb{P}[B^{P_{-D/C}}\cap A\in \mathcal{A}]&\ge\limsup_{n\to\infty} \mathbb{P}[B^{P_n}\cap A\in \mathcal{A}]\\&\ge \liminf_{n\to\infty} \mathbb{P}[ B^{P_n}\cap A\in \mathcal{A}]\ge \mathbb{P}[B^{P_{/C-D}}\cap A\in \mathcal{A}]\nonumber.
\end{align}
These inequalities are in fact true  without the assumption $A\cap F=\emptyset$. Indeed, let $A\subset E$ finite  and $\mathcal{A}$ be an upwardly closed subset of $2^A$. We define $A'=A\backslash F$ and
\[\mathcal{A}'=\{X\subset A'|X\cup(A\cap C)\in \mathcal{A}\}.\]
Note that $\mathcal{A}'$ is upwardly closed subset of $2^{A'}$. 

Then $\mathbb{P}[B^{P_{/C-D}}\cap A\in \mathcal{A}]=\mathbb{P}[B^{P_{/C-D}}\cap A'\in \mathcal{A}']$. Moreover, for any large enough $n$, we have  $\mathbb{P}[B^{P_n}\cap A\in \mathcal{A}]=\mathbb{P}[ B^{P_n}\cap A'\in \mathcal{A}']$. Clearly $A'\cap F=\emptyset$, so we reduced the problem to the already established case.

Choosing $A=\{e\}$ and $\mathcal{A}=\{\{e\}\}$ in (\ref{ineqst}), we get that $\langle P_{/C-D}e,e\rangle\le \langle P_{-D/C}e,e\rangle$ for all $e\in E$. Inequality (\ref{ineqst}) tells us that $B^{P_{-D/C}}$ stochastically dominates $B^{P_{/C-D}}$. But if $\langle P_{/C-D}e,e\rangle=\langle P_{-D/C}e,e\rangle$ for all $e\in E$, then the distribution of $B^{P_{/C-D}}$ and $B^{P_{-D/C}}$ must be the same. Then inequality (\ref{ineqst}) gives the statement.
\end{proof}

Let $A$ be any finite set. We define the martingale $X_n$ by 

\[X_n=\mathbb{P}[A\subset B^P|B^P\restriction F_n]=\mathbb{P}[A\subset B^{P_{/B^P\cap F_n-F_n\backslash B^P}}].\] Combining the previous lemma with our assumptions on $B^P$ we get that with probability $1$ we have $\lim X_n=\mathbb{P}[A\subset B^{P_{/B^P\cap F-F\backslash B^P}}]$. On the other hand we have \[\lim X_n=\mathbb{P}[A\subset B^P|B^P\restriction F].\] The statement follows.

\end{proof}

\begin{lemma}\label{lemmakond}
Let $(G,o,P,F)$ be a unimodular random decorated \wpc\, where $P$ is an orthogonal projection with probability $1$. Then with probability $1$, we have that for any finite set $A\subset V(G)\times K$
\[\mathbb{P}(A\subset B^P|B^P\restriction F)=\mathbb{P}(A\subset B^{P_{/B^P\cap F-F\backslash B^P}}).\]
\end{lemma}
\begin{proof}
From Lemma \ref{lemmakond2}, we have that for all $e\in V(G)\times K$ we have \[\langle P_{/B^P\cap F-F\backslash B^P} e,e\rangle\le \langle P_{-F\backslash B^P/B^P\cap F}e,e\rangle.\] From Lemma \ref{lemmaTr}, we have   $\Tr(G,o,P_{/B^P\cap F-F\backslash B^P})=\Tr(G,o, P_{-F\backslash B^P/B^P\cap F})$, which imply that with probability $1$ we have 
$\langle P_{/B^P\cap F-F\backslash B^P} e,e\rangle= \langle P_{-F\backslash B^P/B^P\cap F}e,e\rangle$  for any $e\in \{o\}\times K$, but then it is true for any $e$ from unimodularity. (See \cite[Lemma 2.3 (Everything Shows at the Root)]{ally}.) Therefore, Lemma \ref{lemmakond1} can be applied to get the statement.

\end{proof}

The lemma above  establishes  Conjecture 9.1 of \cite{lydet} in  the special unimodular case. Note that this conjecture is false in general as it was pointed out to the author by Russel Lyons. Indeed, it follows from the results of Heicklen and Lyons \cite{Heicklen} that  for the WUSF on certain trees, conditioning on all edges but one does not (a.s.) give a measure
corresponding to an orthogonal projection, because the probability of the
remaining edge to be present is in~$(0, 1)$~a.s.

\subsection{Limit of conditional determinantal processes}

\begin{theorem}\label{condlimit}
Let $(G_n,o_n,P_n,C_n,D_n)$ be a convergent sequence of unimodular random decorated \wpcs\ with limit $(G,o,P,C,D)$. Assume  that $P_n$ and $P$ are orthogonal projections  and $(C_n,D_n)$ and $(C,D)$ are all permitted with probability $1$. Then $(G_n,o_n,(P_n)_{/C_n-D_n})$ converges to $(G,o,P_{/C-D})$. 
\end{theorem}

This will follow from applying the next lemma twice, first for the sequence $P_n$, then for $I-(P_n)_{/C}$ with $D_n$ in place of $C_n$. At the second time we need to use Proposition \ref{propfele} to show that the conditions of the lemma are satisfied.

\begin{lemma}\label{lemma23}
Let $(G_n,o_n,P_n,C_n,D_n)$ be a convergent sequence of unimodular random decorated \wpcs\ with limit $(G,o,P,C,D)$. Assume  that $P_n$ and $P$ are orthogonal projections  and $C_n$, $C$ are all independent with probability $1$. Then $(G_n,o_n,(P_n)_{/C_n},D_n)$ converges $(G,o,P_{/C},D)$. 
\end{lemma} 
\begin{proof}
 The presence of $D_n$ does not not add any extra difficulty to the problem, so for simplicity of notation we will prove the following statement instead:

\begin{itshape}
Let $(G_n,o_n,P_n,C_n)$ be a convergent sequence of unimodular random decorated \wpcs\ with limit $(G,o,P,C)$. Assume  that $P_n$ and $P$ are orthogonal projections, $C_n$  and $C$ are all independent with probability $1$. Then $(G_n,o_n,(P_n)_{/C_n})$ converges to $(G,o,P_{/C})$.
\end{itshape}

We start by the following lemma.
\begin{lemma}\label{lemmakorl}
Let $(G_n,o_n,P_n,C_n)$ be a convergent sequence of  decorated \wpcs\ with limit $(G,o,P,C)$. Assume  that $P_n$ and $P$ are orthogonal projections, $C_n$  and $C$ are all independent, and there is an $r$ such that $C_n\subset V(B_r(G_n,o_n))\times K$ and $C\subset V(B_r(G,o))\times K$. Then $(G_n,o_n,(P_n)_{\times C_n})$ converges to $(G,o,P_{\times C})$.
\end{lemma}
\begin{proof}
Let us choose an orthogonal projection $\Pi$ from a small neighborhood $U$ of $P$. If this neighborhood is small enough, then $C$ is independent with respect to $\Pi$. For $c\in C$, we have $\Pi_{\times \{c\}} e=\Pi e -\frac{\langle \Pi e,c \rangle }{\langle \Pi c,c \rangle}\Pi c$. Indeed, clearly $\Pi e -\frac{\langle \Pi e,c \rangle }{\langle \Pi c,c \rangle}\Pi c\in \range \Pi \cap [\{c\}]^\bot$, moreover with the notation $\alpha=\frac{\langle \Pi e,c \rangle }{\langle \Pi c,c \rangle}$ for any $w\in \range \Pi \cap [\{c\}]^\bot$ we have
\[
\left\langle w, e-(\Pi e -\alpha \Pi c) \right\rangle= \left\langle w,(I-\Pi )e \right\rangle+\left\langle w, \alpha \Pi c\right\rangle=\left\langle \Pi w, \alpha c\right\rangle= \left\langle w, \alpha c\right\rangle=0.
\]
By induction we get that 
\[\Pi _{\times C}e=\Pi e-\sum_{c\in C}\alpha_{c,e} \Pi c.\] 
 Here $\alpha_{c,e}$ is a continuous function of $(\langle \Pi x,y\rangle)_{x,y\in C\cup \{e\}}$ in the neighborhood $U$. The statement can be deduced using this. 
\end{proof}

From compactness every subsequence of $(G_n,o_n,P_n,(P_n)_{/C_n},C_n)$ has a convergent subsequence, 
so it is enough to prove the following lemma.
\begin{lemma}\label{lemma25}
Let $(G_n,o_n,P_n,C_n)$ be a convergent sequence of unimodular random decorated \wpcs\ with limit $(G,o,P,C)$. Assume  that $P_n$ and $P$ are orthogonal projections, $C_n$  and $C$ are all independent with probability $1$. If $(G_n,o_n,P_n,(P_n)_{/C_n},C_n)$ converges to $(G,o,P,Q,C)$, then $(G,o,Q)$ has the same distribution as $(G,o,P_{/C})$.
\end{lemma}
\begin{proof}
Using Skorokhod's representation theorem we can find a coupling of  \break $(G_n,o_n,P_n,(P_n)_{/C_n},C_n)$ and $(G,o,P,Q,C)$ such that   $\lim_{n\to\infty} (G_n,o_n,P_n,(P_n)_{/C_n},C_n)= (G,o,P,Q,C)$ with probability $1$. By definition there is a random sequence $r_1,r_2,\dots$ such that $\lim_{n\to\infty} r_n=\infty$ with probability $1$, and there is a root preserving graph isomorphism $\psi_n$ from $B_{r_n}(G,o)$ to $B_{r_n}(G_n,o_n)$ such that $\bar{\psi}_n(C\cap(B_{r_n}(G,o)\times K))=C_n\cap(B_{r_n}(G_n,o_n)\times K)$, where $\bar{\psi}_n(v,k)=(\psi_n(v),k)$ and with probability $1$ for each $e,f\in V(G)\times K$ we have
\[\lim_{n\to\infty} \langle P_n \bar{\psi}_n e, \bar{\psi}_n f  \rangle=\langle Pe,f \rangle, \]
and
\[\lim_{n\to\infty} \langle (P_n)_{/C_n} \bar{\psi}_n e, \bar{\psi}_n f  \rangle=\langle Qe,f \rangle. \]
Of course, $\bar{\psi}_n e$ only makes sense if $n$ is large enough. 

Let us define $C_n(r)= C_n\cap (B_r(G_n,o_n)\times K)$ and $C(r)= C\cap (B_r(G,o)\times K)$. 

Lemma \ref{lemmakorl} gives us that for any $r$ we have

\begin{equation}\label{limitfixr}
\lim_{n\to\infty}  \langle (P_n)_{\times C_n(r)} \bar{\psi}_n(e),\bar{\psi}_n(f)  \rangle=\langle P_{\times C(r)} e,f   \rangle.
\end{equation}

Note that $\range P_{\times C(r)}$ is a decreasing sequence of subspaces with intersection $\range P_{\times C}$. So $P_{\times C(r)}$ converges to $P_{\times C}$ in the strong operator topology. 

In particular, for any $e,f\in V(G)\times K$, we have 
\begin{equation}\label{rtends}
\lim_{r\to\infty} \langle P_{\times C(r)} e,f  \rangle=\langle P_{\times C} e,f  \rangle, 
\end{equation}
and
\begin{equation}\label{rtends2}
\lim_{r\to\infty} \langle (P_n)_{\times C_n(r)} \bar{\psi}_n(e),\bar{\psi}_n(f)  \rangle=\langle (P_n)_{\times C_n} \bar{\psi}_n(e),\bar{\psi}_n(f)   \rangle.
\end{equation}

We need the following elementary fact.
\begin{lemma}
Let  $a(r,n)$ be non-negative real numbers, such that for any fixed $n$, the sequence $a(r,n)$ is monotone decreasing in $r$. Let  $A_n=\lim_{r\to\infty} a(r,n)$, assume that for any fixed $r$ the limit  $B_r=\lim_{n\to\infty} a(r,n)$ exists. Then $\lim_{n\to\infty } A_n\le \lim_{r\to\infty} B_r$ if these limits exist.
\end{lemma}

Note that if $e=f$ then the limits in (\ref{rtends}) and (\ref{rtends2}) are decreasing limits as we observed in Proposition \ref{pertimes}. So the previous lemma  combined with equation (\ref{limitfixr}) gives us that for any $e\in V(G)\times K$ we have
\[\lim_{n\to\infty} \langle (P_n)_{\times C_n} \bar{\psi}_n e,\bar{\psi}_n e  \rangle\le \langle P_{\times C} e,e\rangle.\]
Combining this with Proposition \ref{pertimes}, we get that
\begin{equation}\label{trineq}
\langle Qe,e \rangle=\lim_{n\to\infty} \langle (P_n)_{/ C_n} \bar{\psi}_n e,\bar{\psi}_n e  \rangle\le \langle P_{/ C} e,e\rangle.
\end{equation}
 On the other hand, from Lemma \ref{lemmaTr2}, we know that

\begin{align*}
\mathbb{E} \sum_{k\in K} \langle Q(o,k),(o,k)\rangle&=\Tr (G,o,Q) =\lim_{n\to\infty} \Tr (G_n,o_n,(P_n)_{/C_n})\\&= \lim_{n\to\infty} \Tr (G_n,o_n,P_n)=\Tr(G,o,P)=\Tr(G,o,P_{/C})\\&=\mathbb{E} \sum_{k\in K} \langle P_{/C}(o,k),(o,k)\rangle.
\end{align*}

From this and inequality (\ref{trineq}) we get that $\langle Q(o,k),(o,k) \rangle=\langle P_{/C}(o,k),(o,k) \rangle$ for all \break $k\in K$ with probability $1$, so from unimodularity (\cite[Lemma 2.3 (Everything shows at the root)]{ally}) it follows that 
\begin{equation}\label{diagOK}
\langle Qe,e \rangle=\lim_{n\to\infty} \langle (P_n)_{/ C_n} \bar{\psi}_n e,\bar{\psi}_n e  \rangle=\langle P_{/C}e,e \rangle
\end{equation}
 for all $e\in V(G)\times K$ with probability $1$.

Now we prove that with probability 1  for every $e,f\in V(G)\times K$ we have $\langle Qe,f\rangle =\langle P_{/C}e,f\rangle$. 
This is clear if $e\in C$, because in that case $P_{/C}e=Qe=e$. So assume that $e\not\in C$, then
\begin{align*}
|\langle P_{/C}e,f\rangle-\langle Q e,f\rangle|&=|\langle P_{\times C}e,f\rangle-\langle Q e,f\rangle| \\ &\le |\langle P_{\times C}e,f\rangle-\langle P_{\times C(r)} e,f\rangle|\\&\qquad+|\langle P_{\times C(r)} e,f\rangle-\langle (P_n)_{\times C_n(r)} \bar{\psi}_n e,\bar{\psi}_n f\rangle|\\&\qquad+|\langle (P_n)_{\times C_n(r)} \bar{\psi}_n e,\bar{\psi}_n f\rangle-\langle (P_n)_{\times C_n} \bar{\psi}_n e,\bar{\psi}_n f\rangle|\\&\qquad +
|\langle (P_n)_{\times C_n} \bar{\psi}_n e,\bar{\psi}_n f\rangle-\langle Q e,f\rangle|.
\end{align*}
Pick any $\varepsilon>0$. If we choose a large enough $r$, then $|\langle P_{\times C}e,f\rangle-\langle P_{\times C(r)} e,f\rangle|<\varepsilon$ 
 and $|\langle P_{\times C(r)}e,e\rangle-\langle P_{\times C}e,e\rangle|<\varepsilon$
 from equation (\ref{rtends}). Fix such an $r$. Then if $n$ is large enough $|\langle P_{\times C(r)} e,f\rangle-\langle (P_n)_{\times C_n(r)} \bar{\psi}_n e,\bar{\psi}_n f\rangle|<\varepsilon$ from equation (\ref{limitfixr}), and also $|\langle (P_n)_{\times C_n} \bar{\psi}_n e,\bar{\psi}_n f\rangle-\langle Q e,f\rangle|<\varepsilon$, because of Proposition \ref{pertimes} and the fact that $e\not\in C$. Finally, observing that $(P_n)_{\times C_n(r)}-(P_n)_{\times C_n}$ is an orthogonal  projection, we have
\begin{align*}
|\langle (P_n)_{\times C_n(r)} \bar{\psi}_n e,& \bar{\psi}_n f\rangle-\langle (P_n)_{\times C_n} \bar{\psi}_n e,\bar{\psi}_n f\rangle|\\ &\le \|(P_n)_{\times C_n(r)} \bar{\psi}_n e- (P_n)_{\times C_n} \bar{\psi}_n e\|_2\\&= 
 \sqrt{\langle (P_n)_{\times C_n(r)} \bar{\psi}_n e- (P_n)_{\times C_n} \bar{\psi}_n e,\bar{\psi}_n e\rangle}\\&\le
\Big(|\langle (P_n)_{\times C_n(r)} \bar{\psi}_n e,\bar{\psi}_n e\rangle -\langle P_{\times C(r)}e,e\rangle |\\&\qquad+|\langle P_{\times C(r)}e,e\rangle-\langle P_{\times C}e,e\rangle|\\&\qquad+ |\langle P_{\times C}e,e\rangle-\langle (P_n)_{\times C_n} \bar{\psi}_n e,\bar{\psi}_n e\rangle|\Big)^{\frac{1}{2}}
\end{align*}

Now, for a large enough $n$ we have $|\langle (P_n)_{\times C_n(r)} \bar{\psi}_n e,\bar{\psi}_n e\rangle -\langle P_{\times C(r)}e,e\rangle |<\varepsilon$ from equation (\ref{limitfixr}) and  $|\langle P_{\times C}e,e\rangle-\langle (P_n)_{\times C_n} \bar{\psi}_n e,\bar{\psi}_n e\rangle|=|\langle P_{/ C}e,e\rangle-\langle (P_n)_{/ C_n} \bar{\psi}_n e,\bar{\psi}_n e\rangle|<\varepsilon$ from line (\ref{diagOK}). Finally, $|\langle P_{\times C(r)}e,e\rangle-\langle P_{\times C}e,e\rangle|<\varepsilon$ follows from the choice of $r$. Putting everything together,
$|\langle P_{/C}e,f\rangle-\langle Q e,f\rangle|<3\varepsilon+\sqrt{3\varepsilon}$, so Lemma \ref{lemma25} follows.  
\end{proof}
This completes the proof of Lemma \ref{lemma23} and Theorem \ref{condlimit}.
\end{proof}



%

\section{The proof of Theorem \ref{MainThm}}\label{sec4}

First we observe that we may assume that $|V(G_n)|\to \infty$. If not, then we can take a large $m=m(n)$ and replace $G_n$ with $m$ disjoint copies of $G_n$, and $P_n$ with the $m$ fold direct sum of copies of $P_n$.

Let $(G,P)$ be a finite graph-positive-contraction, where $P$ is an orthogonal projection. Let $m=|V(G)\times L|$. Fix an ordering $e_1,e_2,\dots,e_{m}$ of the element of $V(G)\times L$. \break Let $E_i=\{e_1,e_2,\dots,e_i\}$. For $e\in V(G)\times L$ let $I(e)$ be the indicator of the event that $e\in B^P$.   Let $g(x)=-x\log x-(1-x)\log(1-x)$. Using the chain rule for the conditional entropy and Lemma \ref{fincond} we obtain that

\begin{align*}
h_L(G,P)&=H(I(e_1),I(e_2),\dots, I(e_m))\\&=\sum_{i=0}^{m-1} H(I(e_{i+1})|I(e_1),I(e_2),\dots,I(e_i))\\&=
\sum_{i=0}^{m-1} \sum_{C\subset E_i} \mathbb{P}[B^P\cap E_i=C] g(\mathbb{P}[e_{i+1}\in B^P|B^P\cap E_i=C])\\&=
\sum_{i=0}^{m-1} \mathbb{E} g(\mathbb{P}[e_{i+1}\in B^{P_{/(E_i\cap B^P)-(E_i\backslash B^P)}}])\\&=
\sum_{i=0}^{m-1} \mathbb{E} g(\langle {P_{/(E_i\cap B^P)-(E_i\backslash B^P)}} e_{i+1}, e_{i+1}\rangle).
\end{align*}
Here  expectation is over the random choice of $B^P$. 

Instead of a fixed ordering of $V(G)\times L$ we can choose a uniform random ordering. Taking expectation we get that
\[h_L(G,P)=\sum_{i=0}^{m-1}  \mathbb{E} g(\langle {P_{/(E_i\cap B^P)-(E_i\backslash B^P)}} e_{i+1}, e_{i+1}\rangle),\]
where  expectation is over the random choice of $E_i=\{e_1,e_2,\dots,e_i\}$ and $B^P$. Note that $g(0)=g(1)=0$, so 
\[g(\langle {P_{/(E_i\cap B^P)-(E_i\backslash B^P)}} e, e\rangle)=0
\]
whenever $e\in E_i$. Also note that $e_{i+1}$ is a uniform random element of $(V(G)\times L)\backslash E_i$. From these it follows that if $e$ is a uniform random element of $V(G)\times L$ independent of $E_i$, then
\begin{multline}\label{lelog2}
\frac{m}{m-i} \mathbb{E} g(\langle {P_{/(E_i\cap B^P)-(E_i\backslash B^P)}} e, e\rangle)= \mathbb{E} g(\langle {P_{/(E_i\cap B^P)-(E_i\backslash B^P)}} e_{i+1}, e_{i+1}\rangle)\le \log 2.
\end{multline}
Thus,

\[h_L(G,P)=\sum_{i=0}^{m-1}  \frac{m}{m-i} \mathbb{E} g(\langle {P_{/(E_i\cap B^P)-(E_i\backslash B^P)}} e, e\rangle).\]
 Let $(G,o,P)=U(G,P)$. Then
\[h_L(G,P)=\sum_{i=0}^{m-1}  \frac{m}{m-i} \frac{1}{|L|}\mathbb{E} \sum_{\ell\in L}g(\langle {P_{/(E_i\cap B^P)-(E_i\backslash B^P)}}(o,\ell), (o,\ell)\rangle).\] 
So
\[\frac{h_L(G,P)}{|V(G)|}=\sum_{i=0}^{m-1}  \frac{1}{m-i} \mathbb{E} \sum_{\ell\in L}g(\langle  {P_{/(E_i\cap B^P)-(E_i\backslash B^P)}} (o,\ell),(o,\ell)\rangle).\] 

For $t\in [0,1)$ we define 
\[H_t(G,P)=\mathbb{E} \sum_{\ell\in L}g(\langle {P_{/(E_i\cap B^P)-(E_i\backslash B^P)}}(o,\ell), (o,\ell)\rangle),\]
where $i=\lfloor tm \rfloor$, and $E_i$ is a uniform random $i$ element subset of $V(G)\times L$ independent of $B^P$ and $o$. For $i=0,1,\dots,m-1$, we have
\[
\frac{1}{m-i}\mathbb{E} \sum_{\ell\in L}g(\langle  {P_{/(E_i\cap B^P)-(E_i\backslash B^P)}}(o,\ell),(o,\ell)\rangle)=
\int_{i/m}^{(i+1)/m} \frac{m}{m-\lfloor tm\rfloor}H_t(G,P) dt.\]
Therefore
\begin{equation}\label{integ}
\frac{h_L(G,P)}{|V(G)|}=\int_0^1 \frac{m}{m-\lfloor tm\rfloor}H_t(G,P) dt.
\end{equation}

Let $m_n=|V(G_n)\times L|$. Recall that  we observed at the beginning of the proof that we may assume that $|V(G_n)|\to \infty$. So we assume this.

\begin{lemma}\label{Htlim}
Let $(G_n,P_n)$ be the sequence given in the statement of the theorem. For any $t\in [0,1)$ we have
\[\lim_{n\to\infty} H_t(G_n,P_n)=\mathbb{E} \sum_{\ell\in L}g(\langle {P_{/(E_t\cap B^P)-(E_t\backslash B^P)}} (o,\ell), (o,\ell)\rangle),\]
where $E_t$ is a Bernoulli($t$) percolation of the set $V(G)\times L$ independent of $B^P$. Consequently,
\[\lim_{n\to\infty}\frac{m_n}{m_n-\lfloor t m_n\rfloor}H_t(G_n,P_n)=\frac{1}{1-t} \mathbb{E} \sum_{\ell\in L}g(\langle {P_{/(E_t\cap B^P)-(E_t\backslash B^P)}} (o,\ell), (o,\ell)\rangle).\]
\end{lemma}
\begin{proof}
From Proposition \ref{prop1} we have $(G_n,o_n,P_n,B^{P_n})\to (G,o,P,B^P)$.  It is straightforward to show that $(G_n,o_n,P_n,E_{\lfloor t m_n \rfloor})\to (G,o,P,E_t)$, here $m_n=|V(G)\times L|$ and $E_{\lfloor t m_n \rfloor}$ is a uniform $\lfloor t m_n \rfloor$ element subset of $V(G)\times L$ independent of $B^{P_n}$. Then  it follows that  $(G_n,o_n,P_n,E_{\lfloor t m_n \rfloor},B^{P_n})\to (G,o,P,E_t,B^P)$. But then with the notations $C_n=E_{\lfloor t m_n \rfloor}\cap B^{P_n}$, $C=E_t\cap B^{P}$,  $D_n=E_{\lfloor t m_n \rfloor}\backslash B^{P_n}$ and  $D=E_t\backslash B^{P}$ we have  $(G_n,o_n,P_n,C_n,D_n)\to (G,o,P,C,D)$.  It follows from Theorem \ref{permthm} and Proposition \ref{permprop}, that $(C_n,D_n)$ and $(C,D)$ are all permitted with probability $1$. It is also clear that $(G_n,o_n,P_n,C_n,D_n)$ are unimodular. Thus applying Theorem \ref{condlimit} we get that $(G_n,o_n,(P_n)_{/C_n-D_n})$ converge to $(G,o,P_{/C-D})$. We define the continuous map \break $f:\mathcal{\wpc}\to\mathbb{R}$ as $f(G,o,P) =\sum_{\ell\in L} g(\langle P(o,\ell),(o,\ell)\rangle)$. Then from the definition of weak* convergence we get that \[\lim_{n\to\infty} \mathbb{E}f(G_n,o_n,(P_n)_{/C_n-D_n})=\mathbb{E}f(G,o,P_{/C-D})\]
and this is exactly what we needed to prove.
\end{proof}

From (\ref{lelog2}) we have $\frac{m_n}{m_n-\lfloor m_n\rfloor}H_t(G,P)\le \log 2$ for any $n$ and $t$. So combining equation (\ref{integ}) and Lemma \ref{Htlim} with the dominated convergence theorem we get that
\begin{align}\label{utel}
\lim_{n\to\infty} \frac{h_L(G_n,P_n)}{|V(G_n)|} &=\lim_{n\to\infty}\int_0^1\frac{m_n}{m_n-\lfloor t m_n\rfloor}H_t(G_n,P_n) dt\\& =
\int_0^1 \lim_{n\to\infty}\frac{m_n}{m_n-\lfloor t m_n\rfloor}H_t(G_n,P_n) dt\nonumber\\  &=
\int_0^1 \frac{1}{1-t} \mathbb{E} \sum_{\ell\in L}g(\langle {P_{/(E_t\cap B^P)-(E_t\backslash B^P)}} (o,\ell), (o,\ell)\rangle) dt\nonumber
\\\intertext{$$=\int_0^1 \frac{1}{1-t} \sum_{\ell\in L} \mathbb{P}[(o,\ell) \not\in E_t]\mathbb{E} \left[g(\langle {P_{/(E_t\cap B^P)-(E_t\backslash B^P)}} (o,\ell), (o,\ell)\rangle)\big| (o,\ell)\not\in E_t\right] dt$$}\nonumber
& =
\int_0^1  \sum_{\ell\in L} \mathbb{E} \left[g(\langle {P_{/(E_t\cap B^P)-(E_t\backslash B^P)}} (o,\ell), (o,\ell)\rangle)\big| (o,\ell)\not\in E_t\right]dt.\nonumber
\end{align}
Here we used the law of total expectation and the fact that \break $g(\langle B^{P_{/(E_t\cap B^P)-(E_t\backslash B^P)}} (o,\ell), (o,\ell)\rangle)=0$  whenever $(o,\ell)\in E_t$. Let $c$ be an i.i.d. uniform $[0,1]$ labeling of $V(G)\times L$. Observe that conditioned on the event $(o,\ell)\not\in E_t$ the distribution of $E_t$ is the same as the distribution of $\{e\in V(G)\times L|c(e)<c(o,\ell)\}$ conditioned on $c(o,\ell)=t$. Let $I(e)$ be the indicator of the event $e\in B^{\rest_L P}$. From Lemma \ref{lemmakond} we get for $\ell\in L$
\begin{align*}
\int_0^1 \mathbb{E}& \left[g(\langle {P_{/(E_t\cap B^P)-(E_t\backslash B^P)}} (o,\ell), (o,\ell)\rangle)\big| (o,\ell)\not\in E_t\right]dt\\&=
\int_0^1 \mathbb{E} \left[g(\mathbb{E}(I(o,\ell)|\{I(f)|f\in E_t\}))\big| (o,\ell)\not\in E_t\right]dt\\&=\int_0^1\mathbb{E} \left[g(\mathbb{E}(I(o,\ell)|\{I(f)|c(f)<c(o,\ell)\}))\big| c(o,\ell)=t\right]dt\\&= \mathbb{E} \left[g(\mathbb{E}(I(o,\ell)|\{I(f)|c(f)<c(o,\ell)\}))\right]=\mathbb{E}\bar{h}((o,\ell),\rest_L P).
\end{align*}
Combining this with equation (\ref{utel}) we get Theorem \ref{MainThm}.

\section{Extension of Theorem \ref{MainThm} to positive contractions}\label{sec5}
To state the extension of Theorem \ref{MainThm} we need another tightness notion. 
 Let $K_0\supset K$ be finite. A random \wpc\ $(G_0,o_0,T_0)$ with support set $K_0$ is called an \emph{$K_0$-extension} of the random \wpc\ $(G,o,T)$ with support set $K$, if $(G_0,o_0,\rest_K(T_0))$ has the same distribution as $(G,o,T)$. We say that the extension is \emph{tight} if $T_0$ is an orthogonal projection with probability $1$. A finite graph-positive-contraction $(G_0,T_0)$ with support set $K_0$ is called an \emph{$K_0$-extension} of the finite graph-positive-contraction $(G,T)$ with support set $K$, if $G=G_0$ and $\rest_K T_0=T$. We say that the extension is \emph{tight}, if $T_0$ is an orthogonal projection.  

Given a sequence of finite graph-positive-contractions $(G_n,T_n)$ with support $K$ and a random \wpc\ $(G,o,T)$ with support set $K$, we say that $\lim U(G_n,T_n)=(G,o,T)$ \emph{p-tightly}, if there is a finite $K_0\supset K$ and there are tight $K_0$-extensions $(G_n,P_n)$ of $(G_n,T_n)$ and a tight $K_0$-extension $(G,o,P)$ of $(G,o,T)$  such that $\lim U(G_n,P_n)=(G,o,P)$.

With these definitions we have the following extension of Theorem \ref{MainThm}.
\begin{theorem}
Let $(G_n,T_n)$ be a sequence of finite graph-positive-contractions such that  $\lim U(G_n,T_n)=(G,o,T)$ p-tightly for some random \wpc\  $(G,o,T)$. Then
\[\lim_{n\to\infty} \frac{h_L(G_n,T_n)}{|V(G_n)|}=\bar{h}_L(G,o,T).\] 
\end{theorem}
\begin{proof}
By the definition of tight convergence, there is a finite $K_0\supset K$ and there are tight $K_0$-extensions $(G_n,P_n)$ of $(G_n,T_n)$ and a tight $K_0$-extension $(G,o,P)$  of $(G,o,T)$ such that $\lim U(G_n,P_n)=(G,o,P)$. Note that the distribution of $B^{T_n}$ is the same as \break $B^{P_n}\cap (V(G)\times K)$. So $h_L(G_n,T_n)=h_L(G_n,P_n)$. Similarly, $\bar{h}_L(G,o,T)=\bar{h}_L(G,o,P)$. So from Theorem \ref{MainThm}
\[\lim_{n\to\infty} \frac{h_L(G_n,T_n)}{|V(G_n)|}=\lim_{n\to\infty} \frac{h_L(G_n,P_n)}{|V(G_n)|}=\bar{h}_L(G,o,P) =\bar{h}_L(G,o,T).\] 

\end{proof}

We do not know whether the condition of p-tightness can be replaced with tightness in the theorem above.

Later we will need the following proposition.
\begin{prop}
Let $K\subset K_0$, such that $|K_0|=2|K|$. Any finite graph-positive-contraction  $(G,T)$ has a tight $K_0$-extension $(G,P)$. 
\end{prop}
\begin{proof}
This is well known, see for example \cite[Chapter 9]{lydet}. We include the proof for the reader's convenience. Let $q(x)=\sqrt{x(1-x)}$ on the interval $[0,1]$ and $0$ otherwise. Using functional calculus we can define $q(T)$, for every positive contraction.  Then the block matrix
\[P=\begin{pmatrix}
T&q(T)\\
q(T)&I-T
\end{pmatrix}\]
gives the desired operator.
\end{proof}

The $K_0$-extension given in the previous lemma will be called the \emph{standard $K_0$-extension} of $(G,T)$. The  standard $K_0$-extension of  a random \wpc\ is defined in the analogous way. 

\section{Sofic entropy: The proof of Theorem \ref{SoficThm}}\label{sec6}

\if \sver 1
Note that for any graph $G$ the set of random $\{0,1\}^K$ colorings of $V(G)$ can be identified with the set of random subsets of $V(G)\times K$. In this proof we use the random subset terminology.

As we mentioned in Subsection \ref{subsecmain}, the inequality $h'(B^T)\le \bar{h}(G_\Gamma,e_\Gamma,T)$ is well known, but  we give the proof for completeness. 

Let $G$ be a graph, and $F$ be a random subset of $V(G)\times K$.  
Let $c$ 	be a $[0,1]$ labeling of $V(G)\times K$. For $e\in V(G)\times K$ let $I(e)$ be the indicator of the event that $e\in F$.
  For $(v,k)\in V(G)\times K$ we define
 
\[\bar{h}((v,k),c,F)=H(I(v,k)|\{I(v',k')|c(v',k')<c(v,k)\}).\]

We also define
\[\bar{h}((v,k),F)=\mathbb{E}\bar{h}((v,k),c,F),\]
where $c$ is an i.i.d. uniform $[0,1]$ labeling of $V(G)\times K$. 

Moreover, if $r$ is an integer, then we define
\[
\bar{h}_r((v,k),c,F)=H(I(v,k)|\{I(v',k')|c(v',k')<c(v,k)\text{ and }(v',k')\in B_r(G,v)\times K\})
\]
and   
\[\bar{h}_r((v,k),F)=\mathbb{E}\bar{h}_r((v,k),c,F),\]
where $c$ is an i.i.d. uniform $[0,1]$ labeling of $V(G)\times K$. 

Comparing these definitions with the definitions given in Subsection \ref{subsecmain}, we see that if $F=B^T$ for some positive contraction $T$, then $\bar{h}((v,k),F)=\bar{h}((v,k),T)$. Thus, it is justified the use the same symbol in both cases. 

If $c$ is a $[0,1]$-labeling such that the labels are pairwise distinct and $G$ is finite, the chain rule of conditional entropy gives us
\[H(F)=\sum_{(v,k)\in V(G)\times K} \bar{h}((v,k),c,F).\]
 Taking expectation over $c$ we get that
 \[H(F)=\sum_{(v,k)\in V(G)\times K} \bar{h}((v,k),F).\]
 Or, alternatively,
\[\frac{H(F)}{|V(G)|}=\mathbb{E}\sum_{k\in K} \bar{h}((o,k),F)\]
 where $o$ is a uniform random vertex of $V(G)$. 
 
Combining this with the well known monotonicity properties of conditional entropies, for any integer $r$, we have
\begin{equation*}
\frac{H(F)}{|V(G)|}=\mathbb{E}\sum_{k\in K} \bar{h}((o,k),F)\le \mathbb{E}\sum_{k\in K} \bar{h}_r((o,k),F).
\end{equation*}

Note that $\bar{h}_r((o,k),F)$ only depends on the distribution of $F\cap (B_r(G,o)\times K)$. Therefore, if $F$ is an  $(\varepsilon,r)$ approximation, then we have
\[\frac{H(F)}{|V(G)|}\le\mathbb{E}\sum_{k\in K} \bar{h}_r((o,k),F)\le \sum_{k\in K} \bar{h}_r((e_\Gamma,k),B^T)+\eta_r(\varepsilon) ,\]
where $\eta_r(\varepsilon)$ does not depend on $G$, and $\eta_r(\varepsilon)\to 0$ as $\varepsilon\to 0$. 
In particular, \[H(\varepsilon,r)\le \sum_{k\in K} \bar{h}_r((e_\Gamma,k),B^T)+\eta_r(\varepsilon)\] tending to $0$ with $\varepsilon$ we obtain that \[\inf_{\varepsilon} H(\varepsilon,r)\le \sum_{k\in K} \bar{h}_r((e_\Gamma,k),B^T).\] But we have
\[\lim_{r\to\infty} \sum_{k\in K} \bar{h}_r((e_\Gamma,k),B^T)=\sum_{k\in K} \bar{h}((e_\Gamma,k),B^T).\] 
Thus tending to infinity with $r$ we get \[h'(B^T)\le \sum_{k\in K} \bar{h}((e_\Gamma,k),B^T)=\bar{h}(G_\Gamma,e_\Gamma,T).\]

Now let $G_1,G_2,\dots$ be a sequence of finite $S$-labeled Schreier  graphs Benjamini-Schramm converging to $(G_\Gamma,e_\Gamma)$. Let $K\subset K_0$, such that $|K_0|=2|K|$. Let $P$ be the standard $K_0$-extension of $T$. Then it is clear that $P$ is an invariant operator on $\ell^2(V(G_\Gamma)\times K_0)$. 

\begin{lemma}
There is a sequence of positive contractions $R_n$ on $\ell^2(V(G_n)\times K_0)$ such that $\lim_{n\to\infty} U(G_n,R_n)=(G_\Gamma,e_\Gamma,P)$. Moreover, the spectral measures $\mu_n=\mu_{U(G_n,R_n)}$ weakly converge to $\mu=\mu_{(G_\Gamma,e_\Gamma,P)}=|K|(\delta_0+\delta_1)$. 
\end{lemma}
\begin{proof}
One can easily define a metric $d'$ on $\mathcal{P}(\mathcal{\wpc})$ such that for any \break sequence of positive contractions $R_n$ on $\ell^2(V(G_n)\times K_0)$, we have that \break $\lim_{n\to\infty} d'(U(G_n,R_n),(G_\Gamma,e_\Gamma,P))=0$ if and only if $\lim_{n\to\infty} U(G_n,R_n)=(G_\Gamma,e_\Gamma,P)$ and $\mu_n$ weakly converge to $\mu$.

Thus if the required sequence does not exist, then there is an $\varepsilon>0$ and an infinite sequence $n_1<n_2<\dots$ such that $d'(U(G_{n_i},R_{n_i}),(G_\Gamma,e_\Gamma,P))\ge \varepsilon$ for any $i$ and any positive contractions $R_{n_i}$ on $\ell^2(V(G_{n_i})\times K_0)$.

We will now use the results of Lyons and Thom \cite{lyth}. In their paper they are using ultralimits. However, by passing to a subsequence we may replace ultralimits by actual limits. Thus \cite[Proposition 4.4, Lemma 4.7 and Remark 4.3]{lyth} provide us a subsequence $(m_i)$ of $(n_i)$ and positive contractions $R_{m_i}$ on $\ell^2(V(G_{m_i})\times K_0)$, such that $\lim_{i\to\infty} U(G_{m_i},R_{m_i})=(G_\Gamma,e_\Gamma,P)$ and $\mu_{m_i}$ weakly converge to $\mu$. Indeed, \cite[Proposition 4.4]{lyth} gives us the convergence $\lim_{i\to\infty} U(G_{m_i},R_{m_i})=(G_\Gamma,e_\Gamma,P)$ and \cite[Proposition 4.7]{lyth} is used to make sure $R_{m_i}$ is indeed a positive contraction. Finally, the convergence of spectral measures follows from \cite[Remark 4.3]{lyth}.

Then $\lim_{i\to\infty} d'(U(G_{m_i},R_{m_i}),(G_\Gamma,e_\Gamma,P))=0$, which contradicts to the choice of the subsequence $(n_i)$.

Finally, observe that \[\Tr(G_\Gamma,e_\Gamma,P)=\Tr(G_\Gamma,e_\Gamma,T)+\Tr(G_\Gamma,e_\Gamma,I-T)=|K|,\] so the spectral measure $\mu$ is indeed equal to $|K|(\delta_0+\delta_1)$.
\end{proof}

Note that $R_n$ is not necessary an orthogonal projection. Now we modify $R_n$ slightly to get an orthogonal projection. Let us define 
\[
w(x)=\begin{cases}
x&\text{ for }0\le x<\frac{1}{2},\\
x-1&\text{ for }\frac{1}{2}\le x\le 1
\end{cases}
\]

Note that $w$ is not continuous, but $w^2$ is continuous. Let $(v_i)_{i=1}^{|V(G_n)\times K_0|}$ be an orthonormal basis of $\ell^2(V(G_n)\times K_0)$ consisting of  eigenvectors of $R_n$, such that $R_n v_i=\lambda_i v_i$. Let $w(R_n)$ be the unique operator, such that $w(R_n)v_i= w(\lambda_i)v_i$ for $i=1,2,\dots,|V(G_n)\times K_0|$. 

Then $P_n=R_n-w(R_n)$ will be the orthogonal projection to the span of $\{v_i|\lambda_i\ge \frac{1}2\}$. Moreover,

\begin{align}\label{repr}
\lim_{n\to\infty} \mathbb{E}\sum_{k\in K_0} \|w(R_n)(o,k)\|_2^2&=\lim_{n\to\infty} \mathbb{E}\sum_{k\in K_0} \langle w(R_n)^2 (o,k),(o,k) \rangle\\&=\lim_{n\to\infty} \int_0^1 w^2 d\mu_n=\int_0^1 w^2 d\mu\nonumber \\&=|K|(w^2(0)+w^2(1))=0\nonumber
\end{align}
Here the expectation is over a uniform random vertex $o$ of $V(G_n)$. 
This easily implies that  $U(G_n,R_n)$ and $U(G_n,P_n)$ have the same limit, that is, $\lim U(G_n,P_n)=(G_\Gamma,e_\Gamma,P)$. (Note that in the language of \cite{lyth} the vanishing limit in (\ref{repr}) means that $(R_n)$ and $(P_n)$ represent the same operator.) Now using Theorem \ref{MainThm} we get that

\begin{align*}\lim_{n\to\infty} \frac{H(B^{\rest_K (P_n)})}{|V(G_n)|}&=\lim_{n\to\infty} h_K(G_n,P_n)\\&=\bar{h}(G_\Gamma,e_\Gamma,\rest_K(P))=\bar{h}(G_\Gamma,e_\Gamma,T).
\end{align*}

Now for any $\varepsilon$ and $r$ for large enough $n$ we have that $B^{\rest_K (P_n)}$ is an $(\varepsilon,r)$-approximation of $B^T$, because  $\lim_{n\to\infty} U(G_n,\rest(P_n))=(G_\Gamma,e_\Gamma,T)$. So $\bar{h}(G_\Gamma,e_\Gamma,T)\le h(B^T)$ follows.

Putting everything together we get that $\bar{h}(G_\Gamma,e_\Gamma,T)\le h(B^T)\le h'(B^T)\le\bar{h}(G_\Gamma,e_\Gamma,T)$. So Theorem \ref{SoficThm} follows.
\else

Recall that for any graphs $G$ the set of random $\{0,1\}^K$ colorings of $V(G)$ can be identified with the set of random subsets of $V(G)\times K$. In this proof we use the random subset terminology.

As we mentioned in Subsection \ref{subsecmain} the inequality $h'(B^T)\le \bar{h}(Q,o_Q,T)$ is well known, but for completeness we give the proof anyway. 

Let $G$ be a graph, and $F$ be a random subset of $V(G)\times K$.  
Let $c$ 	be a $[0,1]$ labeling of $V(G)\times K$. For $e\in V(G)\times K$ let $I(e)$ be the indicator of the event that $e\in F$.
  For $(v,k)\in V(G)\times K$ we define
 
\[\bar{h}((v,k),c,F)=H(I(v,k)|\{I(v',k')|c(v',k')<c(v,k)\}).\]

Moreover, we define
\[\bar{h}((v,k),F)=\mathbb{E}\bar{h}((v,k),c,F),\]
where $c$ is an i.i.d. uniform $[0,1]$ labeling of $V(G)\times K$. 

Moreover, if $r$ is an integer we define
\[
\bar{h}_r((v,k),c,F)=H(I(v,k)|\{I(v',k')|c(v',k')<c(v,k)\text{ and }(v',k')\in B_r(G,v)\times K\})
\]
and   
\[\bar{h}_r((v,k),F)=\mathbb{E}\bar{h}_r((v,k),c,F),\]
where $c$ is an i.i.d. uniform $[0,1]$ labeling of $V(G)\times K$. 

Comparing these definitions with the definitions given in Subsection \ref{subsecmain}, we see that if $F=B^T$ for some positive contraction $T$, then $\bar{h}((v,k),F)=\bar{h}((v,k),T)$. Thus, it is justified the use the same symbol in both cases. 

If $c$ is a $[0,1]$-labeling such that the labels are pairwise distinct and $G$ is finite, the chain rule of conditional entropy gives us
\[H(F)=\sum_{(v,k)\in V(G)\times K} \bar{h}((v,k),c,F).\]
 Taking expectation we get that
 \[H(F)=\sum_{(v,k)\in V(G)\times K} \bar{h}((v,k),F).\]
 Or alternatively,
\[\frac{H(F)}{|V(G)|}=\mathbb{E}\sum_{k\in K} \bar{h}((o,k),F)\]
 where $o$ is a uniform random vertex of $V(G)$. 
 
Combining this with the well known monotonicity properties of conditional entropies we get that
\begin{equation}\label{ineqfixr}
\frac{H(F)}{|V(G)|}=\mathbb{E}\sum_{k\in K} \bar{h}((o,k),F)\le \mathbb{E}\sum_{k\in K} \bar{h}_r((o,k),F).
\end{equation}

Fix an $r$ and consider a finite graph $G$  and a random subset $F$ of $V(G)\times K$. Assume that $F$ is an $(\varepsilon,r)$ approximation of $B^T$, then from (\ref{ineqfixr})
\[\frac{H(F)}{|V(G)|}\le \mathbb{E}\sum_{k\in K} \bar{h}_r((o,k),F).\]
Note that $\bar{h}_r((o,k),F)$ only depends on the distribution of $F\cap (B_r(G,o)\times K)$, so from the fact that $F$ is an  $(\varepsilon,r)$ approximation it follows that
\[\frac{H(F)}{|V(G)|}\le\mathbb{E}\sum_{k\in K} \bar{h}_r((o,k),F)\le \sum_{k\in K} \bar{h}_r((o_Q,k),B^T)+\eta_r(\varepsilon) ,\]
where $\eta_r(\varepsilon)$ does not depend on $G$, and $\eta_r(\varepsilon)\to 0$ as $\varepsilon\to 0$. 
In particular, \[H(\varepsilon,r)\le \sum_{k\in K} \bar{h}_r((o_Q,k),B^T)+\eta_r(\varepsilon)\] tending to $0$ with $\varepsilon$ we obtain that \[\inf_{\varepsilon} H(\varepsilon,r)\le \sum_{k\in K} \bar{h}_r((o_Q,k),B^T).\] But we have
\[\lim_{r\to\infty} \sum_{k\in K} \bar{h}_r((o_Q,k),B^T)=\sum_{k\in K} \bar{h}((o_Q,k),B^T).\] 
Thus tending to infinity with $r$ we get \[h'(B^T)\le \sum_{k\in K} \bar{h}((o_Q,k),B^T)=\bar{h}(Q,o_Q,T).\]

Now let $G_1,G_2,\dots$ be a sequence of finite graphs Benjamini-Schramm converging to $(Q,o_Q)$. Let $K\subset K_0$, such that $|K_0|=2|K|$. And let $P$ be the standard $K_0$ extension of $T$. Then it is clear that $T$ is an invariant operator on $\ell^2(V(Q)\times K_0)$. Then it follows from the results of Lyons and Thom \cite{lyth} that there is a sequence of positive contractions $R_n$ on $\ell^2(V(G_n)\times K_0)$ such that $\lim_{n\to\infty} U(G_n,R_n)=(Q,o_Q,P)$. Moreover, the spectral measures $\mu_n=\mu_{U(G_n,R_n)}$ converge weakly to $\mu=\mu_{(Q,o_Q,P)}=|K|(\delta_0+\delta_1)$. Note that $R_n$ is not necessary an orthogonal projection. Now we modify $R_n$ slightly to get an orthogonal projection. Let us define 
\[
w(x)=\begin{cases}
x&\text{ for }0\le x<\frac{1}{2},\\
x-1&\text{ for }\frac{1}{2}\le x\le 1
\end{cases}
\]

Note that $w$ is not continuous, but $w^2$ is continuous. Let $(v_i)_{i=1}^{|V(G_n)\times K_0|}$ be a orthonormal basis of $\ell^2(V(G_n)\times K_0)$ consisting of  eigenvectors of $R_n$, such that $R_n v_i=\lambda_i v_i$. Let $w(R_n)$ be the unique operator, such that $w(R_n)v_i= w(\lambda_i)v_i$ for $i=1,2,\dots,|V(G_n)\times K_0|$. 

Then $P_n=R_n-w(R_n)$ will be the orthogonal projection to the span of $\{v_i|\lambda_i\ge \frac{1}2\}$. Moreover,

\begin{multline*}
\lim_{n\to\infty} \mathbb{E}\sum_{k\in K_0} \|w(R_n)(o,k)\|_2^2=\lim_{n\to\infty} \mathbb{E}\sum_{k\in K_0} \langle w(R_n)^2 (o,k),(o,k) \rangle=\lim_{n\to\infty} \int_0^1 w^2 d\mu_n=\int_0^1 w^2 d\mu=\\|K|(w^2(0)+w^2(1))=0
\end{multline*}
Here the expectation is over a uniform random vertex $o$ of $V(G_n)$. 
This easily implies that if $U(G_n,R_n)$ and $U(G_n,P_n)$ has the same limit, that is, $\lim U(G_n,P_n)=(Q,o_Q,P)$. Now using Theorem \ref{MainThm} we get that

\[\lim_{n\to\infty} \frac{H(B^{\rest_K (P_n)})}{|V(G_n)|}=\lim_{n\to\infty} h_K(G_n,P_n)=\bar{h}(Q,o_Q,\rest_K(P))=\bar{h}(Q,o_Q,T).\]

But for any $\varepsilon$ and $r$ for large enough $n$ we have that $B^{\rest_K (P_n)}$ is an $(\varepsilon,r)$-approximation of $B^T$, because  $\lim_{n\to\infty} U(G_n,\rest(P_n))=(Q,o_Q,T)$. Then $\bar{h}(Q,o_Q,T)\le h(B^T)$ follows.

Putting everything together we get that $\bar{h}(Q,o_Q,T)\le h(B^T)\le h'(B^T)\le\bar{h}(Q,o_Q,T)$.
\fi

\section{Tree entropy}\label{sec7}

Let $G=(V,E)$ be a locally finite connected graph. Choose an orientation of each edge to obtain the oriented graph $\vec{G}$. The \emph{vertex-edge incidence matrix} $A=(a_{ve})$ of $\vec{G}$ is a  $V\times E$ matrix such that
\[
a_{ve}=
\begin{cases}
1&\text{if $e$ enters $v$,}\\
-1&\text{if $e$ leaves $v$,}\\
0&\text{otherwise.}
\end{cases}
\] 

Let $\bigstar=\bigstar(\vec{G})$ be the closed subspace of $\ell^2(E)$ generated by the rows of $A$, 
and let $P_\bigstar$ be the orthogonal projection from $\ell^2(E)$ to $\bigstar$. If $G$ is finite, then the determinantal measure 
corresponding to $P_{\bigstar}$ is the uniform measure on the spanning trees of $G$ \cite{trans}. Let $\tau (G)$ be the number of spanning trees of 
$G$, then $H(B^{P_{\bigstar}})=\log \tau(G)$. If $G$ is infinite, the corresponding determinantal measure is the so-called \emph{wired uniform 
spanning forest}(WUSF) \cite{wusf1,wusf2,wusf3,wusf4}. Note that  in both cases, the resulting measure does not depend on the chosen orientation of $G$.

Given a rooted graph $(G,o)$  and a non-negative integer $k$, let $p_k(G,o)$ be the probability that a simple random walk starting at $o$ is back at $o$ after $k$ steps.  

The following theorem was proved by Lyons \cite{treeent}.
\begin{theorem}
Let $G_n$ be a sequence of finite connected graphs, such that $|V(G_n)|\to\infty$ and their Benjamini-Schramm limit is a random rooted graph $(G,o)$. Then
\[\lim_{n\to\infty} \frac{\log\tau(G_n)}{|V(G_n)|}=\mathbb{E}\left(\log \deg(o)-\sum_{k=1}^{\infty}\frac{1}{k} p_k(G,o)\right).\] 
\end{theorem}

Using our results we can give another expression for the limiting quantity. Let $G$ be a connected locally finite infinite  graph, let $\mathfrak{F}$ be the WUSF of $G$. For $e\in E(G)$ let $I(e)$ be the indicator of the event that $e\in \mathfrak{F}$. Given a $[0,1]$ labeling $c$ of $E(G)$ and an edge $e\in E(G)$ we define
\[\bar{h}(G,e,c)=H(I(e)|\{I(f)|c(f)<c(e)\}),\]
and
\[\bar{h}(G,e)=\mathbb{E} \bar{h}(G,e,c),\]
where the expectation is over the i.i.d. uniform random $[0,1]$ labeling of $G$. Now we state our version of the tree entropy theorem.
\begin{theorem}
Let $G_n$ be a sequence of finite connected graphs, such that $|V(G_n)|\to\infty$ and their Benjamini-Schramm limit is a random rooted graph $(G,o)$. Then
\[\lim_{n\to\infty} \frac{\log\tau(G_n)}{|V(G_n)|}=\frac{1}{2}\mathbb{E}\sum_{e\sim o} \bar{h}(G,e),\] 
where the summation is over the  edges $e$ incident to the root $o$. 
\end{theorem} 
\begin{proof}
Let $(\vec{G},o)$ be the random rooted oriented graph obtained from $(G,o)$ by orienting each edge independently and uniformly to one of the two possible directions. Let $L(\vec{G})$ be the line graph of $\vec{G}$, that is the vertex set of $L(\vec{G})$ is $V(\vec{G})$ and two vertices of $L(G)$ are connected if the corresponding edges in $\vec{G}$ are adjacent. Let $(\vec{G}',o')$ be obtained from $(\vec{G},o)$ by biasing by the degree of the root. Let $e$ be a uniform random edge incident to $o'$. Then $(L(\vec{G}'),e,P_{\bigstar(\vec{G}')})$ will be a random \wpc, which we denote by $(L,e,P)$. (Here the support set $K$ of $(L,e,P)$ is a one element set.) Now there is an orientation $\vec{G}_n$ of $G_n$ such that the Benjamini-Schramm limit of $\vec{G}_n$ is $(\vec{G},o)$. This can be proved by choosing random orientations, and using  concentration results. We omit the details. Let $(L_n,P_n)$ be the finite-graph-contraction $(L(\vec{G}_n),P_{\bigstar(\vec{G}_n)})$. We have the following lemma.
\begin{lemma}
We have $\lim_{n\to\infty}U(L_n,P_n)=(L,e,P)$.
\end{lemma}
\begin{proof}
This can be proved by slightly modifying the argument of the proof of \break\cite[Proposition 7.1]{ally}. 
\end{proof}    
The proof can be finished using Theorem \ref{MainThm}.

\end{proof}

Both Lyons's and our theorem can be extended to edge weighted graphs, but in this case they are about two different quantities. However, these two quantities are closely related as we explain now. Let $G$ be a connected finite graph, and assume that each edge $e$ has a positive weight  $w(e)$. The weight of a spanning tree $T$ is defined as $w(T)=\prod_{e\in T} w(T)$. Let \[Z(G,w)=\sum_{T\text{ is a spanning tree}} w(T)\]
be the  sum of the weights of the spanning
trees of $G$. Let $\mathfrak{F}$ be a random spanning tree of $G$, such that for any spanning tree $T$ we have $\mathbb{P}(\mathfrak{F}=T)=Z(G,w)^{-1} w(T)$. This is again a determinantal process, the only difference compared to the uniform case is that for each edge $e$ we need to multiply the corresponding column of the vertex-edge incidence matrix by $\sqrt{w(e)}$. In fact, this is the way we define the weighted version of the WUSF for infinite graphs. The Shannon entropy $H(\mathfrak{F})$ of $\mathfrak{F}$ is related to $Z(G,w)$ by the identity
\begin{equation}\label{energyentropy}
H(\mathfrak{F})=\log Z(G,w)-\mathbb{E}\log w(\mathfrak{F}).
\end{equation}

Let $(G_n,w_n)$ be a Benjamini-Schramm convergent sequence of weighted connected graphs, such that $|V(G_n)|\to\infty$ and their Benjamini-Schramm limit is a random rooted weighted graph  $(G,o,w)$. Assume that the weights are uniformly bounded away from zero and infinity, that is, there are $0<C_1<C_2<\infty$ such that all the weight are from the interval $[C_1,C_2]$. Then the generalization of Lyons's theorem states that
\[\lim_{n\to\infty} \frac{\log Z(G_n,w_n)}{|V(G_n)|}=\mathbb{E}\left(\log \pi(o)-\sum_{k=1}^{\infty}\frac{1}{k} p_{k,w}(G,o)\right),\]
where $\pi(v)$ is total weight of the edges incident to $v$, and  $p_{k,w}(G,o)$ is defined using the random walk with transition probabilities $p(x,y)=\pi(x)^{-1} w(xy)$ instead of the simple random walk. On the other hand our theorem states that
\[\lim_{n\to\infty} \frac{H(\mathfrak{F}_n)}{|V(G_n)|}=\frac{1}{2}\mathbb{E}\sum_{e\sim o} \bar{h}(G,e,w),\] 
where $\bar{h}(G,e,w)$ is defined as above, but using the weighted version of the WUSF.

These two statements above together  with equation  (\ref{energyentropy}) of course imply that \break $\lim_{n\to\infty} |V(G_n)|^{-1} \mathbb{E}\log w(\mathfrak{F}_n)$ exists. However, there is a more direct proof. It is based on the observation that
\begin{align*}
\frac{\mathbb{E}\log w(\mathfrak{F}_n)}{|V(G_n)|}&=\frac{1}{|V(G_n)|}\sum_{e\in E(G_n)} \mathbb{P}(e\in \mathfrak{F}_n)\log w(e)\\&=\frac{1}{2}\mathbb{E} \sum_{e\sim o} \mathbb{P}(e\in \mathfrak{F}_n)\log w(e),
\end{align*}
where the last expectation is  over a uniform random $o\in V(G_n)$.  
Since we know that the limit of $\mathfrak{F}_n$ is $\mathfrak{F}$, where $\mathfrak{F}$ is the WUSF of the random rooted weighted graph $(G,o,w)$ (see \cite[Proposition 7.1]{ally}) we get that
\[\lim_{n\to\infty} \frac{\mathbb{E}\log w(\mathfrak{F}_n)}{|V(G_n)|}=\frac{1}{2} \mathbb{E} \sum_{e\sim o} \mathbb{P}(e\in \mathfrak{F})\log w(e).\]

Using equation (\ref{energyentropy}), this provides us another formula for the limit  \break $\lim_{n\to\infty}  |V(G_n)|^{-1} \log Z(G_n,w_n)$. Namely,
\[\lim_{n\to\infty} \frac{\log Z(G_n,w_n)}{|V(G_n)|}=\frac{1}{2} \mathbb{E} \sum_{e\sim o} \left(\mathbb{P}(e\in \mathfrak{F})\log w(e)+\bar{h}(G,e,w) \right).\]

\begin{question}
We have seen that if $(G,o)$ is an infinite random rooted graph which is the limit of finite connected graphs, then
\[\mathbb{E}\left(\log \deg(o)-\sum_{k=1}^{\infty}\frac{1}{k} p_k(G,o)\right)  =\frac{1}{2}\mathbb{E}\sum_{e\sim o} \bar{h}(G,e).\]
Is this true for any infinite unimodular random rooted graph?
\end{question}

\begin{appendix}
\section*{Measurability of the polar decomposition}\label{appn} 

The key idea is that we can realize every operator on a single fixed Hilbert-space. This can be done by using the canonical representatives defined by Aldous and Lyons in a slightly different setting. See Section 2 of \cite{ally}. Now we give the details. Let us call a \rwo\ $(G,o,T)$ half-canonical if the following hold. If $G$ is finite then $V(G)=[0,|V(G)|-1]\subset\mathbb{N}$, if $G$ is infinite then $V(G)=\mathbb{N}$, the root $o$ is equal to $0$, moreover $B_r(G,o)=[0,|B_r(G,o)|-1]$ for all $r$. If $G$ is infinite then $T$ is an operator on the Hilbert space $H=\ell^2(\mathbb{N}\times K)$. If $G$ is finite we will still consider $T$ as an operator on $H$ by setting $T(v,k)$ to be $0$ for all $(v,k)\in (\mathbb{N}\backslash V(G))\times K$. Let $\mathcal{HC}$ be the set of half-canonical \rwo s. We endow $\mathcal{HC}$ with a metric as follows. Given two elements $(G_1,0,T_1)$ and $(G_2,0,T_2)$ of $\mathcal{HC}$, their distance is defined as the infimum of $\varepsilon>0$ such that for $r=\lfloor\varepsilon^{-1}\rfloor$ we have that $G_1$ and $G_2$ are the same restricted to the vertices  $[0,r]$, moreover 
\begin{equation*}
|\langle T_1(v,k), (v',k')\rangle-\langle T_2(v,k), (v',k')\rangle|<\varepsilon
\end{equation*} 
for every $v,v'\in [0,r]$ and $k,k'\in K$. Then the obvious map $g:\mathcal{HC}\to\mathcal{\rwo}$ is continuous. The next lemma shows that we can go the other direction too.
\begin{lemma}
There is a measurable map from $f:\mathcal{\rwo}\to\mathcal{HC}$ such that for any $(G,o,T)\in \mathcal{\rwo}$ we have that $(G,o,T)$ and $f(G,o,T)$ are isomorphic as \rwo s. In other words $g\circ f=\text{id}$. 
\end{lemma}
\begin{proof}
The construction given in Section 2 of \cite{ally} can be adopted to this situation. 
\end{proof} 

Let $B(H)$ be the set of bounded linear operators on $H$. We endow $B(H)$ with a measurable structure by considering the coarsest $\sigma$-algebra such that all the $B(H)\to \mathbb{R}$ maps $T\mapsto \langle Te,f\rangle$ are measurable for $e,f\in \mathbb{N}\times K$. We also endow $H$ with the measurable structure coming from the norm $\|\cdot \|_2$. 

\begin{lemma}
Let $T$ be a $B(H)$-valued measurable map, $x$ be an $H$-valued measurable map defined on the same measurable space. Then $Tx$ is an $H$-valued measurable map. 
\end{lemma}  
\begin{proof}
Let $e_1,e_2,\dots$ be an enumeration of $\mathbb{N}\times K$. Then $Tx$ is the pointwise limit of 
\[y_n=\sum_{i=1}^n \left(\sum_{j=1}^n \langle x,e_j\rangle \langle Te_j,e_i\rangle \right)e_i.\]  
Since $y_n$ is measurable, $Tx$ is measurable too.
\end{proof}

This also has the following consequence.
\begin{lemma}\label{szorzat}
Let $S$ and $T$ be $B(H)$-valued measurable maps. Then $ST$ is a $B(H)$-valued measurable map. 
\end{lemma}
\begin{proof}
Let $e,f\in \mathbb{N}\times K$. Then $STe=S(Te)$, here $Te$ is an $H$-valued measurable map, so $S(Te)$ is an $H$-valued measurable map. So $\langle STe,f\rangle$ is measurable.
\end{proof}

Given a $B(H)$-valued measurable map $T$, let $T^+$ be its generalized inverse. Note that $T^+$ is not necessary bounded. We need the following theorem.

\begin{theorem}[\cite{inverz}]\label{invth}
For any $x\in H$ the map $T^+x$ is an $H$-valued measurable map.  
\end{theorem}

From this we obtain the following lemma.
\begin{lemma}
Let $T$ be a $B(H)$-valued measurable map, and let $T=UP$ be its unique polar decomposition, i.e. $P=\sqrt{A^* A}$ and $U=TP^+$. Then $U$ is a $B(H)$-valued  measurable map.
\end{lemma}
\begin{proof}
From Lemma \ref{szorzat} $(A^*A)^n$ is a measurable map for all $n$. Approximating the square root function by polynomials we get that $P$ is a $B(H)$-valued measurable map. The statement follows by combining Theorem \ref{invth} and the argument of the proof of Lemma \ref{szorzat}.  
\end{proof}

Then it is not difficult to prove the following.
\begin{lemma}
For an \rwo\  $(G,o,T)$ let $T=UP$ be the polar decomposition of $T$, then the map $(G,o,T)\mapsto (G,o,U)$ is measurable.
\end{lemma}

\end{appendix}




\section*{Acknowledgements}
The author is grateful to Mikl\'os Ab\'ert for the useful discussions throughout the writing of this paper, to Russel Lyons and the anonymous referee  for their valuable comments. The author was partially supported by the ERC Consolidator Grant 648017.


\bibliographystyle{imsart-number} 

%




\end{document}